\documentclass[a4paper]{article}

\usepackage{listings}

\lstdefinestyle{Python}{
    language        =   Python, 
    basicstyle      =   \zihao{-5}\ttfamily,
    numberstyle     =   \zihao{-5}\ttfamily,
    keywordstyle    =   \color{blue},
    keywordstyle    =   [2] \color{teal},
    stringstyle     =   \color{magenta},
    commentstyle    =   \color{red}\ttfamily,
    breaklines      =   true,   
    columns         =   fixed,  
    basewidth       =   0.5em,
}

\usepackage{amsthm}

\usepackage{appendix}
\usepackage[all]{xy}\usepackage[latin1]{inputenc}        
\usepackage[dvips]{graphics,graphicx}
\usepackage{amsfonts,amssymb,amsmath,color,mathrsfs, amstext}
\usepackage{amsbsy, amsopn, amscd, amsxtra, amsthm,authblk, enumerate}
\usepackage{upref}
\usepackage[colorlinks,
            linkcolor=red,
            anchorcolor=red,
            citecolor=red
            ]{hyperref}

\usepackage{geometry}
\usepackage[displaymath]{lineno}
\usepackage{float}

\usepackage{yhmath}

\usepackage[normalem]{ulem}

\numberwithin{equation}{section}     

\def\b{\boldsymbol}

\def\cS{\mathcal{S}}

\newcommand{\E}{\mathbb{E}}

\newcommand{\bbR}{\mathbb{R}}

\newtheorem*{theorem*}{Theorem}

\newtheorem{theorem}{Theorem}[section]
\newtheorem{lemma}{Lemma}[section]
\newtheorem{assumption}{Assumption}[section]
\newtheorem{proposition}{Proposition}[section]
\newtheorem{remark}{Remark}[section]
\newtheorem{corollary}{Corollary}[section]

\begin{document}

\title{A sharp uniform-in-time error estimate for Stochastic Gradient Langevin Dynamics}
\author[a,c]{Lei Li\thanks{E-mail: leili2010@sjtu.edu.cn}}
\author[b]{Yuliang Wang\thanks{E-mail: YuliangWang$\_$math@sjtu.edu.cn}}
\affil[a]{School of Mathematical Sciences, Institute of Natural Sciences, MOE-LSC, Qing Yuan Research Institute, Shanghai Jiao Tong University, Shanghai, 200240, P.R.China.}
\affil[b]{School of Mathematical Sciences, Institute of Natural Sciences, Shanghai Jiao Tong University, Shanghai, 200240, P.R.China.}
\affil[c]{Shanghai Artificial Intelligence Laboratory}
\date{}
\maketitle

\begin{abstract}
    We establish a sharp uniform-in-time error estimate for the Stochastic Gradient Langevin Dynamics (SGLD), which is a widely-used sampling algorithm. Under mild assumptions, we obtain a uniform-in-time $O(\eta^2)$ bound for the KL-divergence between the SGLD iteration and the Langevin diffusion, where $\eta$ is the step size (or learning rate). Our analysis is also valid for varying step sizes. Consequently, we are able to derive an $O(\eta)$ bound for the distance between the invariant measures of the SGLD iteration and the Langevin diffusion, in terms of Wasserstein or total variation distances. Our result can be viewed as a significant improvement compared with existing analysis for SGLD in related literature.
\end{abstract}







\section{Introduction}\label{intro}

The Stochastic Gradient Langevin Dynamics (SGLD) \cite{welling2011bayesian}, first proposed by Welling and Teh in 2011, has drawn great attention of researchers from various areas, and it shows outstanding performance when dealing with sampling tasks \cite{abbati2018adageo,li2016scalable,patterson2013stochastic}.
As an online algorithm, SGLD incorporates independent white noise into the well-known Stochastic Gradient Descent (SGD), making it effective for sampling tasks. Equivalently, the SGLD algorithm can be also viewed as adding a random batch to the drift term of the Euler-Maruyama scheme for the (overdamped) Langevin diffusion, which is a time-continuous stochastic process that can converge to a target distribution $\pi$ under suitable assumptions. In this paper, we give an optimal estimate for time-discretization error (the distance between SGLD and the Langevin diffusion), and a sharp bound for the sampling error (distance between SGLD and the target distribution $\pi$ in the sampling task) in terms of Wasserstein or total variation distance as a corollary. In detail, letting $\eta$ be the constant time step (or learning rate), we prove that under mild assumptions, the time-discretization error in terms of KL-divergence is $O(\eta^2)$, which is sharp and enhances the results of most existing analyses \cite{raginsky2017non, zou2021faster, farghly2021time, mou2018generalization, chau2021stochastic,zhang2023nonasymptotic}. The result is also valid for varying step sizes. Moreover, the techniques involved in our analysis can effectively address challenges from the random batch and time discretization (see a more detailed discussion below and in Section \ref{sec:mainresults}). These techniques have the potential for further applications to analyze other stochastic processes and algorithms, for instance, a follow-up work for the sharp error estimate of the Random Batch Method (RBM) for large interacting particle system \cite{huang2024mean}.

Let us first explain the details of the SGLD method. Suppose we aim to generate samples from the target distribution $\pi \propto e^{-\beta U}$, where $U: \bbR^d\to \bbR$ is the free energy and  $\beta>0$ is the inverse temperature. One well-known and effective way to sample from $\pi$ is using overdamped Langevin diffusion, whose invariant measure is exactly $\pi$. It is described by the following stochastic differential equation (SDE) in It\^o's sense:
\begin{equation}\label{eq:overdampedlangevin}
    dX = -\nabla U(X)dt + \sqrt{2\beta^{-1}} dW, \quad X|_{t=0}=X_0,
\end{equation}
where $W$ is the Brownian motion in $\mathbb{R}^d$. The practical sampling method is then to solve the SDE above via suitable numerical schemes. After running the numerical simulation for relatively long time, one treats the obtained numerical solution for \eqref{eq:overdampedlangevin} as an approximation for $\pi$. Consider the classical Euler-Maruyama scheme for \eqref{eq:overdampedlangevin}. Given the time step (or learning rate) $\eta_k$ at $k$-th iteration, and denoting $T_k := \sum_{i=0}^{k-1} \eta_i$, the Euler-Maruyama scheme for \eqref{eq:overdampedlangevin}, which is also called the Unadjusted Langevin Algorithm (ULA), iterates as follows:
\begin{equation}\label{eq:ULA}
    \hat{X}_{T_{k+1}} = \hat{X}_{T_k} - \eta_k \nabla U(\hat{X}_{T_k}) + \sqrt{2\beta^{-1}}(W_{T_{k+1}} - W_{T_k}).
\end{equation}
Based on ULA \eqref{eq:ULA}, the key idea of SGLD is to reduce the computation cost by using the random batch when calculate the drift term $-\nabla U$. In various practical tasks such as the Bayesian inference \cite{welling2011bayesian}, people deal with the potential $U(\cdot)$ coming from high dimensional large-scaled data with size $N$, which is usually a large number. In these applications, $U(\cdot)$ is often of the form \begin{equation}\label{eq:unbiasedproperty}
    U(\cdot)=\mathbb{E}_{\xi}\left[U^{\xi}(\cdot)\right],
\end{equation}
which is the expected value of a function depending on some random variable $\xi  \in \mathcal{S}$. Motivated by the ``random mini-batch" idea from the Stochastic Gradient Descent algorithm proposed by Robbins and Monre decades ago \cite{robbins1951stochastic}, the SGLD algorithm replaces the drift $\nabla U(\cdot)=\mathbb{E}_{\xi}\left[\nabla U^{\xi}(\cdot)\right]$ by a random drift $\nabla U^{\xi}(\cdot)$, which is an unbiased estimate of the original drift. Such a modification for \eqref{eq:ULA} is expected to reduce the computational cost, and meanwhile still converge to the (overdamped) Langevin diffusion \eqref{eq:overdampedlangevin} as the time step tends to zero. In practice, one often has $U(x)=U_0(x)+\frac{1}{N}\sum_{i=1}^N U_i(x)$,
and as in SGD, $\xi$ often represents the minibatch of $\{1,\cdots, N\}$. In this case, for fixed batch-size $S$ ($S$ is a determined constant), $\xi$ belongs to the set $\mathcal{S}=\{(a_1,\dots,a_S): a_i (1\leq i \leq S) \text{ are }  $S$ \text{ different random numbers uniformly chosen from} \{1,\dots, N \}\}$. For $\xi = (a_1,\dots,a_S)$, the corresponding unbiased estimate $U^{\xi}$ is
$U^{\xi}(x)=U_0(x)+\frac{1}{S}\sum_{i=1}^S U_{a_i}(x)$.
In this paper, denoting
\begin{equation*}
    b(\cdot) := -\nabla U(\cdot), \quad b^{\xi}(\cdot) := -\nabla U^{\xi}(\cdot),
\end{equation*}
we consider the following general form of SGLD iteration:
\begin{equation}\label{sgld}
    \bar{X}_{T_{k+1}} = \bar{X}_{T_k} - \eta_k b^{\xi_k}(\bar{X}_{T_k}) + \sqrt{2\beta^{-1}}(W_{T_{k+1}} - W_{T_k}).
\end{equation}
Here, $\xi_k$ are i.i.d. sampled random batches chosen at $T_k$ ($k = 0, 1, 2, \dots$). Also, in our analysis, we consider the following continuous interpolation version which coincides with the discrete SGLD at each grid point $T_k$:
\begin{equation}\label{sgld_continuous}
     \bar{X}_t := \bar{X}_{T_k} + (t - T_k)b^{\xi_{k}}(\bar{X}_{T_k}) +  \sqrt{2\beta^{-1}}\left(W_t - W_{T_k} \right), \quad t \in \left[T_k,T_{k+1}\right).
\end{equation}
For any time $t$, denote the time marginal distributions $\rho_t$, $\bar{\rho}_t$ of solutions to \eqref{eq:overdampedlangevin} and \eqref{sgld_continuous}, repsectively. We aim to study the ``distance" between $\rho_t$ and $\bar{\rho}_t$, and we mainly consider the Kullback-Leibler (KL) divergence \cite{kullback1951information}, defined by
\begin{equation}
    D_{KL}(\mu\|\nu) := \int_{\mathbb{R}^d}  \log \frac{d\mu}{d\nu} \, \mu(dx),
\end{equation}
where $\mu$, $\nu$ are probability measures on $\mathbb{R}^d$ that are absolutely continuous with respect to each other. Note that the KL-divergence is not a distance (or metric) in mathematics, since it does not satisfy the triangular inequality.

Recent years has witnessed great theoretical development for SGLD and related algorithms. In \cite{teh2016consistency}, Teh et. al. proved stability and weak convergence to the Langevin diffusion of the SGLD algorithm.  In \cite{zhang2017hitting}, the authors studied the hitting time of SGLD for non-convex optimization problem, where the authors focused on finding a local minimum instead of analyzing the convergence to a target distribution. For sampling-aimed tasks, much SGLD-related research has been done when the target distribution is log-concave \cite{dalalyan2012sparse,dalalyan2019user,dalalyan2017theoretical}. Recently, convergence analysis for SGLD under non-convex settings has received increasing popularity. In \cite{raginsky2017non}, the authors obtained an error bound for the Wasserstein-2 distance between the SGLD iteration and the target distribution of order $O(k\eta+e^{-k\eta})$, where $k$ is the number of iterations and $\eta$ is the constant time step. In \cite{mou2018generalization}, assuming the boundedness of the drift function, the authors proved the error is of order $O(\sqrt{k\eta})$ in terms of Wasserstein-2 distance. In \cite{zou2021faster}, assuming the second-order smoothness of the drift function, the authors obtained an improved bound of order $O(\sqrt{\eta}/S + \eta + (1-\eta)^k)$ in terms of total variation distance, where $S$ is the batch size. In \cite{zhang2023nonasymptotic}, the authors proved an $O(\eta^{\frac{1}{2}} + e^{-k\eta})$ Wasserstein-1 bound and an $O(\eta^{\frac{1}{4}} + e^{-k\eta})$ Wasserstein-2 bound, and obtained a corresponding $O(\eta^{\frac{1}{4}} + e^{-k\eta})$ error bound for the expected objective function value in the optimization task. A similar error bound was obtained for dependent data streams (i.e. one allows some dependency of the random batches $\xi_0,\xi_1,\dots$) in \cite{chau2021stochastic}.
Different from most former work, where the error bounds had an at least polynomial dependency on the time $T$, in \cite{farghly2021time}, the authors obtained a uniform-in-time bound of order $O(\frac{1}{S}+\sqrt{\eta})$ in terms of the expected generalization error (equivalent to the Wasserstein-1 distance). Notably, to the best of our knowledge, so far none of the existing results has obtained better error bounds than $O(\sqrt{\eta})$, in terms of Wassersten-1, Wasserstein-2, or total variation distances. As we shall see later in Section \ref{sec:discussion}, our sampling error bound (distance between SGLD and the target distribution) is of $O(\eta + e^{-k\eta})$
in terms of any these distances above, which can be viewed as a great improvement compared with existing ones.

As discussed below, the error analysis for SGLD in this paper faces two main technical challenges: (1) difficulties arising from the time discretization; (2) difficulties arising from the random batch. Here we also briefly summarize some recent literature on these related topics, some of which would provide inspiration or even guidance to our analysis. On one hand, as introduced near \eqref{eq:unbiasedproperty}, the random batch idea arises first from SGD decades ago \cite{robbins1951stochastic}. 
In recent decades, SGD and its variants \cite{kingma2015adam,daniel2016learning,neelakantan2015adding} have received a great deal of attention when solving high-dimensional tasks in machine learning and data science. Lots of theoretical analysis for SGD has been done by former researchers, including loss landscape of SGD iteration \cite{smith2020generalization,smith2021origin}, its dynamical stability \cite{wu2018sgd} and diffusion approximation \cite{wang2022sgd,hu2019diffusion,feng2019uniform}. The application of the methodology of random mini-batch to specific dynamics could result in cheaper computational cost while preserving the dynamical and statistical properties. Examples include the SGLD algorithm we study in the paper and the random batch methods for interacting particle systems \cite{jin2020random,jin2021random}. Mainly based on the consistency property \eqref{eq:unbiasedproperty} of the random batch and the law of large numbers (in time), convergence has been established for various systems involving the random batch \cite{jin2020random, jin2021convergence}, and such general idea of the random batch may provide help when analyzing SGLD.

On the other hand, recent progress in analyzing the Euler-Maruyama discretization \eqref{eq:ULA} (which is also known as the Unadjusted Langevin Algorithm (ULA), or the Langevin Monte Carlo (LMC)) may also give some inspiration to the error estimate for SGLD. This is because ULA can be viewed as the simpler full-batch version of SGLD, and thus the optimal convergence rate of SGLD would not exceed that of ULA.
In \cite{dalalyan2017theoretical}, mainly using the first order smoothness assumption, the authors showed that $D_{KL}(\hat{\rho}_T\|\rho_T)=O(\eta T d)$ (recall the definitions for $\hat{\rho}_T$, $\rho_T$ near \eqref{sgld_continuous}). This KL-error bound is not optimal since the optimal error bound in terms of Wasserstein-2 distance is of $O(\eta)$ \cite{alfonsi2015optimal,dalalyan2019user}. In \cite{mou2019improved}, the authors improved the KL-divergence $D_{KL}(\hat{\rho}_T\|\rho_T)$ from $O(\eta T d)$ to $O(\eta^2d^2T)$, by assuming some additional conditions including the second-order smoothness of the drift. Some techniques in \cite{mou2019improved} are helpful when dealing with terms arising from the time discretization in our analysis.

We summarize here the main motivation, contributions and novelty of our work. Motivated by recent results on the optimal KL error bound for ULA \cite{mou2019improved} and algorithms involving the random batch \cite{jin2020random,jin2021convergence}, in this paper we prove an optimal error bound for KL-divergence between SGLD and the (overdamped) Langevin diffusion. Our result is based on mild assumptions for the drift function $b(\cdot)$ in \eqref{eq:overdampedlangevin} (recall that $b(\cdot) = -\nabla U(\cdot)$) and the corresponding target distribution $\pi \propto e^{-\beta U}$. Basically, we need uniform-in-batch 1st, 2nd order Lipschitz conditions, the distance dissipation condition ($x \cdot b^{\xi}(x) \lesssim -|x|^2 + 1$), and a boundedness assumption ($\mathrm{esssup}_{\xi}\, ||b^{\xi} - b||_{L^{\infty}(\mathbb{R}^d)} < \infty$) for the drift function. Also, in order to make the error bound uniform-in-time, and to make some parts of our proof (for instance, the estimation for Fisher information) more elegant, we assume a warm start condition and a log-Sobolev inequality for the target distribution $\pi$. See more details in Assumption \ref{ass:b}, \ref{ass:pi} at the beginning of Section \ref{sec:mainresults}. Based on these assumptions, our main result can be roughly stated as follows (the rigorous statement of this result is Theorem \ref{longtimesgld_etak} below):
\begin{theorem*}
Consider the probability density functions $\rho_t$, $\bar{\rho}_t$ for $X_t$, $\bar{X}_t$ defined in \eqref{eq:overdampedlangevin}, \eqref{sgld_continuous} with constant time step $\eta$. Suppose Assumptions \ref{ass:b}, \ref{ass:pi} hold. Then for small $\eta$,
\begin{equation}\label{eq:informalthm}
    \sup_{k}D_{KL}(\bar{\rho}_{T_k}\| \rho_{T_k}) \leq C \eta^2.
\end{equation}
where $C$ is a positive time-independent constant.
\end{theorem*}
Note that it is possible to prove a time-dependent bound for the discretization error $D_{KL}\left(\bar{\rho}_t \| \rho_t \right)$ without assuming the LSI for the target distribution $\pi$ (see more detailed discussion after Assumption \ref{ass:pi} in Section \ref{sec:mainresults} below).  Our result is valid for varying time steps, and we also give explicit dependence of the dimension $d$ and the inverse temperature $\beta$. See more details in Section \ref{sec:mainresults} below. Moreover, the novelty of our work can be concluded as follows:
\begin{itemize}
    \item In terms of the result, under the non-convex settings, our error bound (1) is optimal for the time step $\eta$ and is valid for the varying time step case; (2) is uniform-in-time; (3) has a regular dependency on other parameters such as the dimension $d$ and the inverse temperature $\beta$.
    \item In terms of the methodology, unlike regular approaches like the coupling method in many existing papers analyzing stochastic algorithms, especially those involving the random batch, our analysis is based on distances between \textbf{distributions} instead of \textbf{trajectories}. This enables us to improve the error bound to the optimal one (see more detailed discussions in Section \ref{sec:discussion}). Moreover, we have successfully overcome challenges from (1) the time-discretization and (2) the random batch, via a sequence of probabilistic tools. In detail, we study the time-discretization error (the main part is from $J_2$ in \eqref{eq4_9} below) from a Fokker-Planck equation for \eqref{sgld_continuous} and further calculations via Bayes' formula and integral by parts. The error term ($J_3$ in \eqref{eq4_9}) from the random batch is handled using Girsanov's transform, combined with some properties of the path measures. Such techniques may have broad applications when analyzing other random algorithms, such as tamed schemes for SDEs and random batch methods for interacting particle systems.
\end{itemize}

To end this section, we will give a brief discussion of why our analysis can achieve the sharp error bound. 
The convergence nature of the SGLD algorithm is based on the consistency of the random batch and the Markov property. Indeed, the discrete SGLD process in our setting is a Markov Chain \cite{durrett2019probability}, and it is a time-homogeneous one if the step size (or learning rate) is constant.  Let $\rho_t^{\b{\xi}}$ be the probability density of the fixed-batch version of SGLD \eqref{sgld_continuous} for a given sequence of batches $\b{\xi}:=(\xi_0, \xi_1, \cdots, \xi_k, \cdots)$ (which is actually the ULA with the drift $b^{\xi}(x)$). 
Consequently, from the consistency of the random batch (recall \eqref{eq:unbiasedproperty}), we have the following expression of the density: 
\begin{equation}\label{barpiMC}
    \bar{\rho}_t(\cdot) = \mathbb{E}_{\xi}[\rho_t^{\b{\xi}}(\cdot)].
\end{equation}
Here, $\mathbb{E}_{\xi}[\rho_t^{\b{\xi}}(\cdot)]$ means taking expectation for all possible choices of batch $\b{\xi}$. This property \eqref{barpiMC} then enables us to derive a Fokker-Planck equation for the density $\bar{\rho}_t$ of SGLD, based on the known Fokker-Planck equation for ULA (see Section \ref{sec:ULA} below for more details). By the Markov property, we are able to find that
\begin{equation*}
\mathbb{E}\left[\rho_t^{\b{\xi}} \Big| \xi_i,\,i\geq k\right] = \cS_{T_k,t}^{\xi_k} \bar{\rho}_{T_k}, \quad t \in [T_k,T_{k+1}),
\end{equation*}
where the operator $\cS_{T_k,t}^{\xi_k}$ is the evolution operator of \eqref{sgld_continuous} from $T_k$ to $t$. Now, with the explicit Fokker-Planck equations for both SGLD and the Langevin diffusion, we are able estimate the KL-divergence between $\bar{\rho}_t$ and $\rho_t$, and obtain a sharp error bound after a sequence of careful calculations and advanced tools.

As will be discussed in details in Section \ref{subsec:moredis}, our direct estimate for the distribution (instead of analysis for trajectories in many other coupling methods) is crucial to obtain the sharp error bound. In fact, in algorithms involving random batches (such as the SGD and the random batch methods), for each single step, the error of the drift is $O(1)$. 
Such algorithms are often proved to have a strong error 
like (see for example \cite{jin2020random})
\[
\sqrt{\E\|X-\bar{X}\|^2} \sim \sqrt{(e^{k\eta}-1)\eta}.
\]
Clearly, the strong error decays like $\sqrt{\eta}$, which is actually sharp. As mentioned in \cite{jin2020random},
the averaging effect in time ensures a convergence like "law of large numbers" in time so the convergence rate is $\sqrt{\eta}$. 
The strong error estimate can imply that 
\[
(\E_{\xi} W_2^2(\rho_t^{\b{\xi}}, \rho_t))^{1/2} \sim \sqrt{(e^{k\eta}-1)\eta}.
\]
This indicates that the fluctuation of the trajectory and $\rho_t^{\b{\xi}}$ is really like $\sqrt{\eta}$. Hence, the existing error estimates of SGLD based on the trajectory estimates can achieve $\sqrt{\eta}$ at most. However, for suitable distance $d(\cdot,\cdot)$, the true goal is actually $d(\mathbb{E}_{\xi} \rho_t^{\b{\xi}}, \rho_t)$ instead of $\mathbb{E}_{\xi}d( \rho_t^{\b{\xi}}, \rho_t)$. Therefore, it is possible to achieve a sharp error bound by directly studying the evolution of distributions.

The detailed analysis begins with Fokker-Planck equations for both continuous-time Langevin diffusion and its random-batch numerical scheme - SGLD. We refer to Section \ref{sec:localpf} for more technical details and we provide a high-level overview of the proof sketch in Section \ref{sec:overview}.

The rest of the paper is organized as follows. Before our main results, we  introduce some preliminaries in Section \ref{sec:pre} including the basic background and properties of the (overdamped) Langevin diffusion and its Euler-Maruyama discretization, ULA. In Section \ref{sec:mainresults}, we present and prove our main result: a sharp uniform-in-time estimate for SGLD, and some useful corollaries are also included. In Section \ref{sec:localpf}, we provide the proof for a crucial local estimate, which is omitted in Section \ref{sec:mainresults} for the convenience of readers.  Some technical details that are not so important are moved to the Appendix. In Section \ref{sec:discussion}, we perform discussion on our analysis on  Wasserstein and total variation distances between the SGLD iteration and the target distribution $\pi$ based on our main results. The ergodicity of SGLD algorithm as well as the reason why our analysis is ``sharp" are also discussed in Section \ref{sec:discussion}.

\section{Preliminaries}\label{sec:pre}
For the convenience of readers, let us begin with some basic definitions and properties that would be useful in this paper.

\subsection{(Overdamped) Langevin Diffusion}
Derived from the Langevin equation, the (overdamped) Langevin diffusion has strong, well-known physical background \cite{pavliotis2014stochastic}.
Let us first consider the following second-order SDE called Langevin equation:
\begin{equation}\label{langevin_eq}
    \ddot{X} = -\nabla U(X) - \gamma \dot{X} + \sqrt{2\gamma / \beta}\, \dot{W},
\end{equation}
where $\gamma$ is the friction coefficient, $\beta^{-1} := k_B T$ with $k_B$ being the Boltzmann constant and $T$ being the (physical) temperature. The term $\gamma \dot{X}$ here describes the linear dissipation damping, and the term $\sqrt{2\gamma / \beta} \dot{W}$ describes the fluctuation stochastic forcing. Note that the Brownian Motion $W$ is not pointwise differentiable, and so $\dot{W}$ should be understood in the distribution sense and is in fact the white noise \cite{nualart2006malliavin}. Moreover, thanks to the Fluctuation-Dissipation Theorem \cite{kubo1966fluctuation}, we are able to relate the dissipation and the fluctuation by considering the covariance of the fluctuation term.

Langevin equation \eqref{langevin_eq} describes a particle  moving according the Newton's second law and  being in contact with  a heat reservoir that is at equilibrium at temperature $T$. This physical is called an open classical system. To see this more clearly, we rewrite the Langevin equation \eqref{langevin_eq} into an SDE system. Letting $Y = \dot{X}$, we have
\begin{equation}
\left\{
    \begin{array}{ll}
        dX = Ydt, & \\
        dY = -\nabla U(X) dt - \gamma Y dt + \sqrt{2\gamma \beta^{-1}} dW . &
    \end{array}
\right.
\end{equation}

It is also well-known that if we consider the time rescaling $X^{\gamma}(t) = X(\gamma t)$, and let $\gamma \rightarrow +\infty$, the Langevin equation \eqref{langevin_eq} turns into the overdamped Langevin diffusion \cite{pavliotis2014stochastic}:
\begin{equation}
    dX = -\nabla U(X)dt + \sqrt{2\beta^{-1}} dW.
\end{equation}

The derived (overdamped) Langevin diffusion above has some pretty properties. First, it has invariant distribution $\pi \propto e^{-\beta U}$. Second, the invariant distribution $\pi$ satisfies the detailed balance since the probability flux $J(\pi) := b\cdot \pi - \frac{1}{2} \nabla \cdot(\Sigma \pi)$ equals zero with $b = -\nabla U$ and $\Sigma = \sqrt{2\beta^{-1}}I$ here, and thus the stationary diffusion is reversible \cite{pavliotis2014stochastic}. Third, the density of the Langevin diffusion satisfies the following  Fokker-Planck equation:
\begin{equation}\label{FP}
    \partial_t \rho_t = -\nabla \cdot (\rho_t b) +\beta^{-1} \Delta \rho_t, \quad \rho_t|_{t=0} = \rho_0.
\end{equation}
with $b = -\nabla U$.
Moreover, if we add some stronger assumptions to the potential $U$ like strongly convexity, the Langevin diffusion above is guaranteed to converge to equilibrium \cite{markowich2000trend}. Note that in this paper, instead of log-concaveness of the invariant distribution $\pi$, we add the much weaker assumption that  $\pi$ satisfies a Log-Sobolev inequality (LSI).


\subsection{Unadjusted Langevin Algorithm (ULA)}\label{sec:ULA}
Among the classical numerical schemes for the Langevin diffusion \eqref{eq:overdampedlangevin}, the simplest one is the well-known Euler-Maruyama scheme, which is also known as the Unadjusted Langevin Algorithm (ULA) or the Langevin Monte Carlo (LMC). Consider this Euler-Maruyama scheme  for \eqref{eq:overdampedlangevin}
\begin{equation}\label{ula}
    \hat{X}_{T_{k+1}} = \hat{X}_{T_k} + \eta_k b(\hat{X}_{T_k}) + \sqrt{2\beta^{-1}\eta_k}Z_k,\quad Z_k \sim N(0,I_d) \quad i.i.d. ,
\end{equation}
where $\eta_k$ is the  time step at $k$-th iteration and $T_k := \sum_{i=0}^{k-1} \eta_i$. \eqref{ula} naturally admits the following continuous version:
\begin{equation}\label{ula_continuous}
    \hat{X}_t := \hat{X}_{T_k} + (t - T_k)b(\hat{X}_{T_k}) +  \sqrt{2\beta^{-1}} \left(W_t - W_{T_k} \right), \quad t \in \left[T_k,T_{k+1}\right),
\end{equation}
where $W_t$ is the Brownian Motion in $\mathbb{R}^d$. We also denote $\hat{\rho}_t(x)$  the probability density function of  $\hat{X}_t$. It is not difficult to show  that $\hat{\rho}_t$ satisfies a Fokker-Planck equation as in the following lemma:
\begin{lemma}\label{lmm:ulafp}
Denote $\hat{\rho}_t(x)$  the probability density function of  $\hat{X}_t$ defined in \eqref{ula_continuous}. Then the following Fokker-Planck equation holds:
\begin{equation}\label{FPhat}
    \partial_t \hat{\rho}_t = -\nabla \cdot (\hat{\rho}_t \hat{b}_t) + \beta^{-1} \Delta \hat{\rho}_t, \quad \hat{\rho}_t|_{t=0} = \rho_0,
\end{equation}
where
\begin{equation}
    \hat{b}_t(x) := \mathbb{E}\left[b\left( \hat{X}_{T_k}\right) |\hat{X}_t = x\right], \quad t \in \left[T_k,T_{k+1}\right).
\end{equation}
\end{lemma}

The derivation is not difficult and has appeared in many classical work \cite{mou2019improved}. Indeed, for $t \in \left[T_k, T_{k+1} \right)$, consider the random variable $\hat{\rho}_t|\hat{\mathcal{F}}_{T_k}$, where $\hat{\mathcal{F}}_{T_k} = \sigma (\hat{X}_s,s\leq T_k)$. Then, by definition of $\hat{X}_t$, $\hat{\rho}_t|\hat{\mathcal{F}}_{T_k}$ satisfies:
\begin{equation}
    \partial_t (\hat{\rho}_t|\hat{\mathcal{F}}_{T_k}) = -\nabla \cdot \left( b(\hat{X}_{T_k})(\hat{\rho}_t|\hat{\mathcal{F}}_{T_k})\right) + \beta^{-1} \Delta (\hat{\rho}_t|\hat{\mathcal{F}}_{T_k}), \quad t \in \left[T_k, T_{k+1} \right).
\end{equation}
Taking expectation, we have
\begin{equation}
    \mathbb{E}\left[\partial_t (\hat{\rho}_t|\hat{\mathcal{F}}_{T_k})\right] = \partial_t \hat{\rho}_t,\quad 
    \mathbb{E}\left[\Delta (\hat{\rho}_t|\hat{\mathcal{F}}_{T_k}) \right] = \Delta \hat{\rho}_t,
\end{equation}
and for the drift term,
\begin{equation}
    \begin{aligned}
    \mathbb{E}\left[-\nabla \cdot \left( (\hat{\rho}_t|\hat{\mathcal{F}}_{T_k})(x) b(\hat{X}^\xi_{T_k})\right) \right] & =-\nabla \cdot \int (\hat{\rho}_t|\hat{\mathcal{F}}_{T_k})(x|y)b(y)\hat{\rho}_{T_k}(y)dy\\
    &= -\nabla \cdot \int b(y) \hat{\rho}_{t, T_k}(x, y) dy\\
    & = -\nabla \cdot \left( \hat{\rho}_t(x)\mathbb{E} \left[ b(\hat{X}^\xi_{T_k}) | \hat{X}_t = x \right]\right),
    \end{aligned}
\end{equation}
where $\hat{\rho}_{t, T_k}$ indicates the joint distribution density of $\hat{\rho}_t$ and $\hat{\rho}_{T_k}$.
Note that we use Bayes' law  in the second equality.
Combining all the above, we obtain the desired result \eqref{FPhat}.

\section{Main results}\label{sec:mainresults}
In this section, we present our main result of this paper: a sharp uniform-in-time error estimate for SGLD. In Section \ref{sec:ass}, we present the assumptions required by later analysis. In Section \ref{auxiliary}, we give some auxiliary results, including propagation of LSI of Langevin diffusion, moment control for ULA, and estimation of the Fisher information for ULA. After these preparations, in Section \ref{sec:mainthm}, we present the main theorem, Theorem \ref{longtimesgld_etak}, and a crucial local result, Proposition \ref{local_estimate}.

\subsection{Assumptions}\label{sec:ass}
Before the main results and proofs, we firstly give the following assumptions we use throughout this paper. Recall the definition of the drift function $b$ in SGLD and its target distribution $\pi$ in Section \ref{intro}.  For technical reasons, we will need the following assumptions.
\begin{assumption}\label{ass:b}
\quad
\begin{itemize}
    \item[(a)] (1st order smoothness) For a.e. $\xi \sim \nu$, $b^{\xi}$ is  Lipschitz with a uniform constant $L$, i.e. $\forall \xi$,  $\forall x,y \in \mathbb{R}^d$,
    \begin{equation}
        |b^{\xi}(x) - b^{\xi}(y)| \leq L |x - y|.
    \end{equation}
    \item[(b)] (2nd order smoothness) For a.e. $\xi \sim \nu$, $\nabla b^{\xi}$ is  Lipschitz with a uniform constant $L_2$, i.e. $\forall \xi$, $\forall x,y \in \mathbb{R}^d$,
    \begin{equation}\label{eq:lipJacob}
        |\nabla b^{\xi}(x) - \nabla b^{\xi}(y)| \leq L_2 d^{-\frac{1}{2}} |x - y|.
    \end{equation}
    \item[(c)] (distance dissipation) For a.e. $\xi \sim \nu$, $b^{\xi}$ is confining in the sense that 
    \begin{equation}
        x \cdot b^{\xi}(x) \leq -\mu |x|^2 + \sigma
    \end{equation}
    with $\mu$, $\sigma$ being positive constants.
    \item[(d)] (boundedness) $b^{\xi} - b$ is uniformly bounded, namely,
    \begin{equation}\label{eq:bddiff}
        \mathrm{esssup}_{\xi}\, ||b^{\xi} - b||_{L^{\infty}(\mathbb{R}^d)} < + \infty.
    \end{equation}
\end{itemize}
Moreover, the coefficient $L$, $L_2$, $\mu$, $\sigma$  are independent of $\xi$ and $d$.
\end{assumption}

In condition (a), the Lipschitz constant $L$ is the upper bound of the spectrum norm (largest singular value) of the Jacobian matrix and it is insensitive to the dimension $d$. 
In \eqref{eq:lipJacob} in condition (b), 
we have put $d^{-1/2}$ in the constant and assumed $L_2$ to be independent of $d$. The intuition is that the left hand side is the spectral norm of the Jacobian matrix, which can be assumed to be insensitive to $d$  while $|x-y|$ scales like $\sqrt{d}$. It is in fact unnatural to assume this for every $x, y$ to hold. It will be more natural to assume $\E|\nabla b^{\xi}(X) - \nabla b^{\xi}(Y)|^2 \leq L_2^2 d^{-1}\E |X - Y|^2$ for random variables $X, Y$ whose distributions are close to being isotropic. Such an assumption, however, requires the distribution of the arguments and is thus difficult to use. The simple assumption \eqref{eq:lipJacob} shall be reasonable in the average sense and therefore we believe it captures the correct scaling.
As we shall later, putting such a factor can make the bounds of the relative entropy and Fisher information linear in $d$ (see our discussion at the end of Section \ref{subsec:moredis}), which is the correct scaling. This is a difference from the result in \cite{mou2019improved}.

Note that condition (d) is equivalent to saying that $b^{\xi}-b^{\tilde{\xi}}$ is also uniformly bounded for two different $\xi$, $\tilde{\xi}$. 
In (d) of Assumption \ref{ass:b}, we do not need the uniform boundedness for every $b^{\xi}$, though (d) naturally holds when $\mathrm{esssup}_{\xi} \|b^{\xi}(x)\|_{\infty} < \infty$, which is a much stronger condition. Actually the condition (d) is natural in applications. In many problems in machine learning, the loss has the same asymptotics for different data point. For example, people often put some regularization term of the form $R_{\lambda}(x)=\frac{\lambda}{2}|x|^2$. This then gives the same asymptotic behavior far away. Moreover, condition (d) of Assumption \ref{ass:b} does not necessarily require the uniform boundedness of $b$ (although in some earlier papers regarding SGLD, the uniform boundedness condition for the drift is assumed, see for instance, Theorem 1 of \cite{mou2018generalization}). Take the classical Bayesian inference as an example. The drift term under this setting is in fact $b(\cdot) = \nabla \log p(\cdot) + \frac{1}{N} \sum_{i=1}^N \nabla \log p(x_i | \cdot)$, where $p(\cdot)$ is the prior, and $x_1,\dots,x_i$ are data items. Then, the random batch approximation in \cite{welling2011bayesian} is $b^\xi =  \nabla \log p(\cdot) + \frac{1}{N} \sum_{i=1}^S \nabla \log p(x_{a_i} | \cdot)$, where $S$ is the batch size, and $(a_1,\dots, a_S)$ is a random subset of $(1,\dots, N)$. Clearly, in this case, the condition (d) does not assume any restriction on the prior distribution $p(\cdot)$.

In our paper, condition (c) is only used  to control the moments. Actually condition (c) is the confinement condition which is crucial in literature for ergodicity. In our case, the log Sobolev inequality later in Assumption \ref{ass:pi} actually plays the role of confinement for ergodicity.   Moreover, it is not difficult to see that if one has the stronger boundedness condition $\mathrm{esssup}_{\xi} \sup_x |b^{\xi}(x)| < \infty$, then conditions (c) and (d) can be removed since the only place we use the distance dissipation condition (c) is the moment control (Proposition \ref{thm:moment}), but the only place we use the moment control is bounding terms like $b^{\xi}(X_t)$. Now if we assume each $b^{\xi}$ is bounded, we do not need the result for moment control any more.

\begin{assumption}\label{ass:pi}
\quad
\begin{itemize}
    \item[(a)] ($\lambda$-warm start) There exists $\lambda \geq 1$ such that the initial distribution $\rho_0$ satisfies
    \begin{equation}
        \frac{1}{\lambda} \leq \frac{\rho_0}{\pi} \leq \lambda,
    \end{equation}
    where $\pi \propto  e^{-\beta U}$ is the invariant distribution of \eqref{eq:overdampedlangevin}. Note that $\rho_0$ is the initial distribution for all the processes $X$, $\hat{X}$, and $\bar{X}$, and $\lambda$ is independent of the dimension $d$.
     \item[(b)] (LSI for $\pi$) The invariant distribution $\pi\propto e^{-\beta U}$ satisfies a Log-Sobolev inequality with a constant $ \kappa(\beta)$, namely, $\forall f \in C_{\ge 0}^{\infty}$,
    \begin{equation}
    Ent_{\pi} (f) := \int f \log f d\pi -\left(\int f d\pi\right) \log\left(\int f d\pi\right) \leq \kappa(\beta) \int \frac{|\nabla f|^2}{f} d\pi.
    \end{equation}
    Here, $C_{\ge 0}^{\infty}$ consists of all nonnegative smooth functions, and we set $|\nabla f|^2/f = 0$ if $f = 0$.
\end{itemize}
\end{assumption}

The Log-Sobolev inequality (LSI), first discussed by Gross in 1975 \cite{gross1975logarithmic}, is a necessary assumption to establish our uniform-in-time result. Besides the simplest Gaussian distribution, other distributions satisfying an LSI can be found following the Bakry-Emery criterion \cite{bakry1985diffusions}, including strongly log-concave ones. It is also shown that distributions that are strongly log-concave outside a compact set also satisfy the LSI \cite{ledoux1999concentration}. 
We remark that the log-concaveness outside a compact set can imply both the distance dissipation and the log Sobolev inequality so the LSI assumption is much weaker than the log-concaveness assumption in many other literature. However, the distance dissipation condition $x\cdot b\le -\mu |x|^2+\sigma$ outside a compact set can not derive the corresponding LSI so neither the distance dissipation nor the LSI assumption can be simply removed though they seem to have repetition somehow. Note that the distance dissipation condition, which is weaker than convexity, is essential when controlling the moment of SDE solutions, and is a commonly assumed condition in other related results \cite{zou2021faster, mou2019improved}. 

We also remark here that it is possible to prove a time-dependent error bound for $D_{KL}\left(\bar{\rho}_t \| \rho_t \right)$ if we remove Assumption \ref{ass:pi}. In fact, tracking our proof, conditions (a), (b) in Assumption \ref{ass:pi} are only used when (1) proving a uniform log-Sobolev inequality for $\rho_t$ (density for the Langevin diffusion); (2) making the error bound uniform-in-time via the log-Sobolev inequality for $\rho_t$ at \eqref{LSIapply}; (3) providing an $O(1)$-upper bound for the Fisher information in Lemma \ref{Mfisher}. Note that without Assumption \ref{ass:pi}, since we are still assuming the smoothness of the drift, the estimation for Fisher information is still possible via other calculation approaches, such as analysis based on the heat kernel \cite{mou2019improved}, or based on the Stam's convolution equality \cite{stam1959some, huang2024mean}


\subsection{Some auxiliary results}\label{auxiliary}
Before presenting the main theorem, we give some useful auxiliary results first, including the propagation of LSI for the Langevin diffusion \eqref{eq:overdampedlangevin},  moment control for ULA \eqref{ula_continuous}, and estimation of Fisher information for ULA \eqref{ula_continuous}.

\subsubsection{Propagation of Log-Sobolev inequality}
In this section, we aim to establish a Log-Sobolev inequality (LSI) for $\rho_t$, which is the density of $X_t$ in the Langevin diffusion \eqref{eq:overdampedlangevin}. 

\begin{proposition}\label{LSIpropagation}
Consider the Fokker-Planck equation \eqref{FP} associated with the SDE \eqref{eq:overdampedlangevin}. Suppose  Assumption \ref{ass:pi} holds. Then, $\rho_t$ satisfies a LSI with a uniform LSI constant $\lambda^2C_{\pi}^{LS}(\beta)$.
\end{proposition}



The following Holley-Stroock perturbation lemma \cite{bakry2014analysis} is essential to our proof :
\begin{lemma}[Holley-Stroock perturbation]\label{lmm2}
Suppose the probability measure $\nu \in \mathcal{P}(\mathbb{R}^d)$ satisfies an LSI with constant $\kappa$, and $\mu \in \mathcal{P}(\mathbb{R}^d)$ satisfies $d\mu = h d\nu$ with the uniform bounds $\frac{1}{\lambda} \leq h \leq \lambda$. Then $\mu$ satisfies an LSI with constant $\lambda^2\kappa$.
\end{lemma}


\begin{proof}[Proof of Proposition \ref{LSIpropagation}:]
\begin{equation}
    \rho_t = \frac{\rho_t}{\pi}\pi =: q_t \pi.
\end{equation}
Since $b = \beta^{-1}\nabla \log \pi$, the detailed balanced is satisfied \cite{markowich2000trend}. So it is well-known that  $q_t$ satisfies the following backwards Kolmogorov equation \cite{li2020large}:
\begin{equation}\label{qt}
    \partial_t q_t = b \cdot \nabla q_t + \beta^{-1} \Delta q_t, \quad q_t|_{t=0}=\frac{\rho_0}{\pi},
\end{equation}
and $q_t$ can be expressed by the following conditional expectation form:
\begin{equation}\label{qt_E}
    q_t(x) = \mathbb{E}\left[q_0(X_t) | X_0 = x\right].
\end{equation}
Then, as a direct consequence of \eqref{qt_E}, we know that at any time $t>0$, $\inf_x q_t \geq \inf_x q_0$, and $\sup_x q_t \leq \sup_x q_0$.
Combining this fact and condition (a) in Assumption \ref{ass:pi}, it holds that
\begin{equation}\label{qtbound}
    \frac{1}{\lambda} \leq q_t \leq \lambda, \quad \forall t \geq 0.
\end{equation}
Then, combining \eqref{qtbound} , Lemma \ref{lmm2}, and (b) in Assumption \ref{ass:pi}, the conclusion holds.

\end{proof}

Note that the fact $\inf_x q_t \geq \inf_x q_0$, $\sup_x q_t \leq \sup_x q_0$ can also be derived from classical results of PDE. Indeed, since the backward Kolmogorov equation \eqref{qt} for $q_t$ is of parabolic type, the maximal principle holds, whose details can be found in Chapter 7 of \cite{evans2022partial}. By maximal principle and condition (a) in Assumption \ref{ass:pi}, we also obtain
\begin{equation}
     \frac{1}{\lambda} \leq q_t \leq \lambda, \quad \forall t \geq 0.
\end{equation}

\subsubsection{Moment control}
In this section, we aim to find a uniform-in-time bound for the moments $\mathbb{E}|\bar{X}^\xi_t|^p$ with $\bar{X}^{\xi}$ defined in \eqref{barpiMC} of Section \ref{intro} and $p=2,4$. The details could be found in the proposition below:

\begin{proposition}\label{thm:moment}
Consider the process $\bar{X}_t$ defined in \eqref{sgld_continuous} of Section \ref{intro}. Suppose (a), (c) of Assumption \ref{ass:b} holds. For all integer $p \geq 1$, there exists positive constants $c_p$ and $\Delta_p$ independent of $t$ and $\b{\xi}$ such that if $\eta_k \leq \Delta_p$ for all $k$, then
\begin{equation}\label{momentbdd}
    \mathbb{E}\left[|\bar{X}_t|^p \Big| \mathcal{F}_{\xi}\right]\leq c_p d^{\frac{p}{2}}\left(1 + \beta^{-\frac{p}{2}} \right), \quad \forall t \geq 0.
\end{equation}
Moreover, for all integer $p \geq 1$, there exists a   positive constant $\alpha_p$  such that when $\alpha < \alpha_p$,
\begin{equation}
    \sup_{t\geq 0} \mathbb{E}\left[e^{\alpha|\bar{X}_t|^p}\Big| \mathcal{F}_{\xi}\right] < + \infty. 
\end{equation}
Here, $\mathcal{F}_{\xi}$ denotes the $\sigma$-algebra generated by $\b{\xi}=(\xi_k)_{k\ge 0}$.
\end{proposition}
See Section \ref{omit_momentcontrol} for the detailed proof.


\subsubsection{Estimate of the Fisher information}\label{subsec:fisher}
The Fisher information for a probability measure $\rho$ is defined by 
\begin{equation}
    \mathcal{I}(\rho) := \int_{\mathbb{R}^d} |\nabla \log \rho|^2 \rho\, dx.
\end{equation}
In our analysis, we come up with the exponential-weighted sum of Fisher information at the grid point $T_k$ of ULA (SGLD with fixed batch), and a uniform-in-time $O(1)$ bound   for this summation is  required. To handle this, our strategy is to firstly estimate the continuous sum (i.e. integration) of the exponential-weighted Fisher information in Lemma \ref{Mfisher}. Then, using the existing stability result (Lemma \ref{fisher_regularity} below) for Fisher information of ULA \cite{mou2019improved}, we are able to control the desired discrete summation with the integration of the exponential-weighted Fisher information along the time axis.

Recall that $\rho_t^{\b{\xi}}$ is the law of the SGLD \eqref{sgld_continuous} conditioning on the given sequence of batches $\b{\xi}=(\xi_0, \xi_1, \cdots, \xi_k, \cdots)$. In Section \ref{omit_Mfisher}, we first establish
\begin{equation}\label{eqq315}
    \frac{d}{dt} D_{KL}(\rho^{\b{\xi}}_t\|\pi) \leq -c_0 \beta^{-1} \mathcal{I}(\rho^{\b{\xi}}_t) + c_1d(\beta+1),
\end{equation}
where we recall $\pi \propto e^{-\beta U}$ and consequently
\begin{equation}\label{eqq316}
    D_{KL}(\rho^{\b{\xi}}_t\|\pi) \leq cd(\beta+1)\beta^{-1}\kappa(\beta)^{-1},
\end{equation}
where $c_0$, $c_1$, $c$ are positive constants independent of $t$ and $\xi$. Then, after simple calculation including integration by parts, we are able to prove the following Lemma:
\begin{lemma}\label{Mfisher}
 Let $f$ be a non-increasing, nonnegative, piecewise-constant function. Then under Assumption \ref{ass:b}, \ref{ass:pi}, there exists $\Delta_0>0$, for any $A_0>0$, there exist positive constants $M_1$, $M_2$,
independent of $T$ and $\b{\xi}$ such that  when the step size sequence $\{\eta_k\}$ defined in \eqref{sgld} is bounded by $\Delta_0$, the followings hold:
\begin{equation}\label{eq:klbatch}
    \sup_{t\geq 0} D_{KL}(\rho_t^{\b{\xi}}\|\pi) \leq cd(\beta+1)\beta^{-1}\kappa(\beta)^{-1},
\end{equation}
where $c$ is independent of $T$, $d$ and $\b{\xi}$. Moreover, for any fixed $T > 0$,
\begin{multline}\label{eqq318}
    \int_0^T e^{-A_0(T-s)}f(s)\mathcal{I}(\rho^{\b{\xi}}_s)ds
    \leq M_1 \beta D_{KL}(\rho_0\|\pi)f(0) e^{-A_0T}\\
    + M_2d(\beta +1) (A_0\beta^{-1}\kappa(\beta)^{-1}+\beta)\int_0^T e^{-A_0(T-s)}f(s)ds
\end{multline}
\end{lemma}

The next lemma bounds the Fisher information $\mathcal{I}(\rho^{\b{\xi}}_t)$ at time $t$ with the Fisher information $\mathcal{I}(\rho^{\b{\xi}}_{T_k})$ at the grid point. The estimation is valid for small $\eta_k$, and we can view it as one kind of stability for the Fisher information of an only piecewise continuous process. The proof of the following lemma can be found in Lemma 6 of \cite{mou2019improved}

\begin{lemma}\label{fisher_regularity}
Under the same setting of Lemma \ref{Mfisher}, if $\eta_{k} \leq \frac{1}{2L}$, $k\ge 0$, then there is a positive constant $c$ independent of $k$, $d$, $\beta$ and $\b{\xi}$ such that
\begin{equation}\label{eqq319}
    \mathcal{I}(\rho^{\b{\xi}}_{t}) \leq 8\mathcal{I}(\rho^{\b{\xi}}_{t_0}) + c\eta_{k}^2d,\quad \forall T_k \le t_0\le t\le T_{k+1}.
\end{equation}

\end{lemma}

Clearly, with a factor $d^{-1/2}$ in the Lipschitz constant in Assumption \ref{ass:b}, the bounds of the relative entropy and Fisher information are linear in $d$. The proofs of Lemma \ref{Mfisher} and Lemma \ref{fisher_regularity} can be found in Appendix \ref{sgldpf}.

\subsection{Main theorems: sharp uniform-in-time error analysis for SGLD}\label{sec:mainthm}

Equipped with the preparation work before, we are then able to establish a sharp uniform-in-time second-order error estimate for SGLD in terms of KL-divergence. Our analysis for SGLD is from local to global. The following local estimation is crucial to our main result, and the $\eta^2$ term here is the core of our $O(\eta^2)$ bound in the main theorem. To avoid overloading the notation, the dependence on the size of function family $N$ and other parameters in Assumption \ref{ass:b}, \ref{ass:pi} are implicit. 

\begin{proposition}\label{local_estimate}
Consider the probability density functions $\rho_t$, $\bar{\rho}_t$ for $X_t$, $\bar{X}_t$ defined in \eqref{eq:overdampedlangevin}, \eqref{sgld_continuous}. Then under Assumption \ref{ass:b}, \ref{ass:pi}, there exist $A_0 := 1 / (2\lambda^2\kappa(\beta)\beta)$ and positive constants  $A_1$, $\Delta_0$ independent of $k$, $\beta$ and $d$ such that when $\eta_k<\Delta_0$, the following local estimation holds:
\begin{multline}\label{eq:prop32}
    \frac{d}{dt} D_{KL}(\bar{\rho}_t\|\rho_t) \leq - A_0 D_{KL}(\bar{\rho}_t\|\rho_t) \\
    + A_1  \eta_k^2 \left(d \left(\beta^3 + \beta^{-1} + 1\right) + (\beta + \beta^{-1}+ 1) \mathcal{I}(\bar{\rho}_{T_k})\right),\quad \forall t \in [T_k,T_{k+1}).
\end{multline}
\end{proposition}
Note that we omit the terms like $\beta^2$ in \eqref{eq:prop32} to simplify the expression. In fact, $\beta^2 \lesssim \beta^3 + \beta^{-1}$ due to Young's inequality.

For the convenience of readers, we move the proof of Proposition \ref{local_estimate} to the next section. Combining this local estimation for SGLD and the previous result for estimating the Fisher information of ULA (fixed batch SGLD), it is not difficult to obtain the following main theorem. 

Define the following piecewise-constant function $f$:
\begin{equation}\label{eq:f}
    f(t) := \sum_{i=0}^{\infty} \eta_i^2\textbf{1}_{[T_{i},T_{i+1})}(t),
\end{equation}
where $\textbf{1}_{E}$ is the indicator function for set $E$.

\begin{theorem}\label{longtimesgld_etak}[Sharp error analysis for SGLD]
Consider the probability density functions $\rho_t$, $\bar{\rho}_t$ for $X_t$, $\bar{X}_t$ defined in \eqref{eq:overdampedlangevin}, \eqref{sgld_continuous}. Suppose the time step sequence $\{\eta_k\}$ is non-increasing. Then under Assumption \ref{ass:b}, \ref{ass:pi}, there exist $A_0 := 1 / (2\lambda^2\kappa(\beta)\beta)$ and positive constants $\Delta_0$, , $c_1$, $c_2$ independent of $k$, $\beta$ and $d$ such that when $\eta_0\leq \Delta_0$,
\begin{equation}
    D_{KL}(\bar{\rho}_{T_k}\| \rho_{T_k}) \leq c_1(\beta^2+1)\eta_0^2 e^{-A_0T_k} + c_2dg(\beta) \int_0^{T_k}e^{-A_0(T_k-s)}f(s)ds,
\end{equation}
for $f$ defined in \eqref{eq:f} and $g(\beta) := \beta^3+\beta^{-1} + 1 + A_0^2 (\beta^2 + \beta^{-1})$.
Consequently, if the non-increasing time step sequence $\{\eta_k\}$ converges to zero, and $\sum_{i=0}^{+\infty} \eta_i = + \infty$, the KL-divergence $D_{KL}(\bar{\rho}_{T_k}\| \rho_{T_k})$ also tends to zero.
\end{theorem}

\begin{remark}
We write $g(\beta) = \beta^3+\beta^{-1} + 1 + A_0^2 (\beta^2 + \beta^{-1})$ just to match the result when the potential $U$ is strongly convex. In fact, for the strong convexity case, the constant $\kappa(\beta)$ in LSI scales as $\kappa(\beta) \sim \beta^{-1}$ (\cite{bakry1985diffusions}). Consequently, $A_0$ is of $O(1)$, and $g(\beta) \sim \beta^3 + \beta^{-1}$.
\end{remark}

\begin{proof}[Proof of Theorem \ref{longtimesgld_etak}:]
By Proposition \ref{local_estimate}, since $\eta_k \leq \eta_0 \leq \Delta_0$, and since the Fisher information $\mathcal{I}(\rho)$ is a convex functional with respect to $\rho$ (see \cite[Lemma 4.2]{fournier2014propagation}), for $t \in [T_k, T_{k+1})$ we have
\begin{equation}
\begin{aligned}
    \frac{d}{dt} D_{KL}(\bar{\rho}_t\|\rho_t) &\leq - A_0  D_{KL}(\bar{\rho}_t\|\rho_t) + A_1 \eta_k^2 \left(d(\beta^3+\beta^{-1}) + (\beta+\beta^{-1})\mathcal{I}(\bar{\rho}_{T_{k}}) \right)\\
    &\leq - A_0 D_{KL}(\bar{\rho}_t\|\rho_t) + A_1 \eta_k^2 \left(d(\beta^3+\beta^{-1}) + (\beta+\beta^{-1})\mathbb{E}_{\xi}\mathcal{I}(\rho_{T_{k}}^{\b{\xi}}) \right).
\end{aligned}
\end{equation}
Then by Gr\"onwall's inequality, for all $k$, we have
\begin{equation}\label{equation342}
\begin{aligned}
    &\quad D_{KL}(\bar{\rho}_{T_k}\| \rho_{T_k}) \leq \mathbb{E}_{\xi} \sum_{i=0}^{k-1} \int_{T_i}^{T_{i+1}} e^{-A_0  (T_k - s)} A_1  \eta_i^2 \left(d(\beta^3+\beta^{-1}+1)+(\beta+\beta^{-1}+1)\mathcal{I}(\rho^{\b{\xi}}_{T_i}) \right)ds\\
    & = A_1d(\beta^3+\beta^{-1}+1) \int_0^{T_k} e^{-A_0(T_k-s)} f(s) ds
     + A_1(\beta+\beta^{-1}+1) \mathbb{E}_{\xi} \int_0^{T_{k}} e^{-A_0(T_k-s)}\mathcal{I}(\rho_{T_{i(s)}}^{\b{\xi}}) f(s) ds\\
    & \leq  \tilde{A}_1 d(\beta^3+\beta^{-1}+1) \int_0^{T_k} e^{-A_0(T_k-s)} f(s) ds
    + \bar{A_1} (\beta+\beta^{-1}+1)  \mathbb{E}_{\xi} \int_0^{T_k} e^{-A_0(T_k-s')}\mathcal{I}(\rho_{s'}^{\b{\xi}}) f(s') ds'
\end{aligned}
\end{equation}
and the last inequality of \eqref{equation342} is due to Lemma \ref{fisher_regularity}. The notation $i(s)$ above means the largest $i$ such that $T_i \leq s$. Also, $\tilde{A_1}$ is different from $A_1$ due to the constant $\mathcal{I}(\rho^{\b{\xi}}_0)$, and $\bar{A_1}$ is changed from $A_1$ by multiplying a constant coming from Lemma \ref{fisher_regularity} and $e^{A_0 \eta_0}$. The small constant $e^{A_0 \eta_0}$ appears due to the following two facts: (1) from Lemma \ref{fisher_regularity}, $\mathcal{I}(\rho_{T_{i(s)}}^{\b{\xi}})$ is controlled by $\mathcal{I}(\rho_{s}^{\b{\xi}})$ for $s \in [T_{i(s)-1}, T_{i(s)})$ (instead of $[T_{i(s)}, T_{i(s)+1})$; (2) the sequence $(\eta_k)_k$ is non-increasing.

Clearly, $f$ is a piecewise constant, non-increasing, nonnegative function. Then by Lemma \ref{Mfisher}, we are able to estimate the exponential-weighted sum of Fisher information along the path:
\begin{multline}\label{eq345}
\int_0^{T_k} e^{-A_0(T_k-s)}\mathcal{I}(\rho_{s}^{\b{\xi}}) f(s) ds
\leq M_1'\beta\eta_0^2e^{-A_0T_k}\\
+ M_2'(\beta+1)(A_0\beta^{-1}\kappa(\beta)^{-1}+\beta)d\int_0^{T_k}e^{-A_0(T_k-s)}f(s)ds.
\end{multline}
Here, the positive constants $M_1'$, $M_2'$ are independent of $d$, $\beta$, the batch $\b{\xi}$, and the time $T_k$.

Finally, combining \eqref{eq345} and \eqref{equation342}, we have
\begin{multline}\label{eq_325}
    D_{KL}(\bar{\rho}_{T_k}\| \rho_{T_k}) \leq c_1(\beta^2+1)\eta_0^2 e^{-A_0T_k}\\
    + c_2d(\beta^3+\beta^{-1} + 1 + A_0^2 (\beta^2 + \beta^{-1})) \int_0^{T_k}e^{-A_0(T_k-s)}f(s)ds,
\end{multline}
where $c_1 $, $c_2 $,  $A_0 $ are positive constants independent of the iteration number $k$, the inverse temperature $\beta$ and the dimension $d$.

Clearly, by \eqref{eq_325}, when the sequence $\{\eta_k \}$ decays to zero, and  $\sum_{k=0}^{\infty}\eta_k = +\infty$, the KL-divergence $D_{KL}(\bar{\rho}_{T_k}\| \rho_{T_k})$ tends to zero. This then ends the proof.
\end{proof}

As a direct corollary, if we consider the constant step size (or learning rate) case ($\eta_k \equiv \eta$), the following sharp time-independent estimation can be established: (recall the definitions of $A_0$ and $g(\beta)$ in Theorem \ref{longtimesgld_etak}).

\begin{corollary}[Sharp uniform-in-time error analysis for SGLD, constant step size $\eta$]\label{longtimesgld}
Consider the probability density functions $\rho_t$, $\bar{\rho}_t$ for $X_t$, $\bar{X}_t$ defined in \eqref{eq:overdampedlangevin}, \eqref{sgld_continuous}. Suppose $\eta_k=\eta$, $\forall k$. Then, under Assumption \ref{ass:b}, \ref{ass:pi}, there exist positive constants $c$ , $\Delta_0$ independent of $t$, $d$, $\beta$ such that for all $\eta \in (0, \Delta_0)$,
\begin{equation}
    \sup_{t>0} D_{KL}(\bar{\rho}_t\|\rho_t) \leq cd(\beta^2+1 + g(\beta)/A_0)\eta^2.
\end{equation}

\end{corollary}

Since ULA can be viewed as a special case of SGLD when $b^{\xi}\equiv b$ (or the batch size $S$ equals $N$ for the special case $U^{\xi}(x)=U_0(x)+\frac{1}{S}\sum_{i=1}^S U_{a_i}(x)$),  we have the following direct corollary:
\begin{corollary}[Sharp uniform-in-time error analysis for ULA, constant time step $\eta$]
Consider the probability density functions $\rho_t$ and $\hat{\rho}_t$ defined in \eqref{eq:overdampedlangevin}, \eqref{ula_continuous}. Suppose $\eta_k=\eta$, $\forall k$. Then, under Assumption \ref{ass:b}, \ref{ass:pi}, there is  positive constants $c$, $\Delta_0$ independent of $t$, $d$, $\beta$ such that for all $\eta \in (0, \Delta_0)$,
\begin{equation}
    \sup_{t>0} D_{KL}(\hat{\rho}_t\|\rho_t) \leq cd(\beta^2+1 + g(\beta)/A_0)\eta^2.
\end{equation}

\end{corollary}

Based on our main result, Theorem \ref{longtimesgld_etak} or Corollary \ref{longtimesgld}, we are able to obtain the following corollaries, which are useful in many practical tasks.

First, we choose special case for the decaying time step sequence $\{\eta_k\}_{k=0}^{+\infty}$, and applying Theorem \ref{longtimesgld_etak}, we obtain the following asymptotic convergence rate:

\begin{corollary}[asymptotic rate for special case]\label{rateforspecial}
Suppose Assumption \ref{ass:b}, \ref{ass:pi} hold. Set $\eta_i =(\ell+i)^{-\theta}$, $i\in \mathbf{N}$, with $\theta \in (0,1)$ being the decaying coefficient. Here, $\ell$ is a positive integer such that $\eta_0 \leq \Delta_0$, where $\Delta_0>0$ is a positive constant obtained in Theorem \ref{longtimesgld_etak}. Then there exist $k_0>0$, $c>0$ such that $\forall k>k_0$,
\begin{equation}
    D_{KL}(\bar{\rho}_{T_k}\|\rho_{T_k}) \leq cdg(\beta)\left(\frac{1}{k}\right)^{2\theta},
\end{equation}
\end{corollary}

\begin{proof}[Proof of Corollary \ref{rateforspecial}:]

For $\theta \in (0,1)$, by Theorem \ref{longtimesgld_etak}, we have
\begin{equation*}
    \begin{aligned}
    D_{KL}(\bar{\rho}_{T_k}\|\rho_{T_k}) &\leq c_1 (\beta^2+1) e^{-A_0T_k} + c_2dg(\beta)\left(\sum_{i=1}^{\lfloor k/2 \rfloor} + \sum_{i=1+\lfloor k/2 \rfloor}^{k}\right)\\
    &\quad  \left((\ell+i)^{-2\theta} \left(e^{-A_0(T_k - T_i)} - e^{-A_0(T_k - T_{i-1})} \right)\right)\\
    &\leq c_1(\beta^2+1) e^{-A_0T_k} + c_2'dg(\beta)\left(e^{-A_0(T_k - T_{\lfloor k/2 \rfloor})} - e^{-A_0T_k} \right)\\
    &\quad + c_2'g(\beta)d \left(\frac{2}{k}\right)^{2\theta}\left(1 - e^{-A_0(T_k - T_{1+\lfloor k/2 \rfloor})} \right)\\
    &\leq c_1(\beta^2+1) e^{-A_0T_k} + c_2'dg(\beta)e^{-A_0(T_k - T_{\lfloor k/2 \rfloor})} + c_2'd g(\beta) \left(\frac{2}{k}\right)^{2\theta}.
    \end{aligned}
\end{equation*}
As $k\rightarrow +\infty$, $T_k \sim \sum_{i=1}^k i^{-\theta} \sim k^{1-\theta}$. Hence, $e^{-A_0T_k}$ and $e^{-A_0(T_k - T_{\lfloor k/2 \rfloor})}$ decay much faster than $k^{-2\theta}$ as $k \rightarrow +\infty$. Therefore, there exists $k_0 > 0 $ such that when $k > k_0$,
\begin{equation}
    D_{KL}(\bar{\rho}_{T_k}\|\rho_{T_k}) \leq cdg(\beta)\left(\frac{1}{k}\right)^{2\theta},
\end{equation}
where $c$ is a positive constant independent of $k$.
This is what we want.

\end{proof}

\section{Delayed proof for the local estimation}\label{sec:localpf}
In this section, we prove the crucial local analysis result of this paper, Proposition \ref{local_estimate}. For convenience, some technical details are moved to the Appendix.

\subsection{Methods of analysis: an overview}\label{sec:overview}
The proof of our key analysis, Proposition \ref{local_estimate}, is a bit complicated and technical. For readers' convenience, here we first give a high-level overview of our proof, along with some discussions on key techniques leading to the improved error bound compared with existing results.

For simplicity, here we only consider the constant time step $\eta_k \equiv \eta$. Also, we omit the dependence on dimension $d$ and inverse temperature $\beta$ in the error bounds. Then the goal for the local analysis (Proposition \ref{local_estimate}) is to show that
\begin{equation}\label{eq:goal_overview}
    \frac{d}{dt} D_{KL}\left(\bar{\rho}_t \| \rho_t \right) \leq A_0 D_{KL}\left(\bar{\rho}_t \| \rho_t \right) + A_1 \eta^2 \left(1 + \mathcal{I}(\bar{\rho}_{T_k}) \right),\quad \forall t \in [T_k,T_{k+1})
\end{equation}
for some positive $A_0$, $A_1$ independent of $\eta$ and $T$. Clearly, once \eqref{eq:goal_overview} is proved, by Gr\"onwall's inequality and the upper bound for the Fisher information  $\mathcal{I}(\bar{\rho}_{T_k})$ in Lemma \ref{Mfisher}, the uniform-in-time $O(\eta^2)$ estimate
\begin{equation*}
    \sup_{t \geq 0}D_{KL}\left(\bar{\rho}_t \| \rho_t \right) \lesssim \eta^2
\end{equation*}
is established, which is the simplified version of Theorem \ref{longtimesgld_etak}. We prove \eqref{eq:goal_overview} by the following main steps:
\begin{enumerate}
    \item Estimate time derivative of KL-divergence $\frac{d}{dt}D_{KL}(\bar{\rho}_t \| \rho_t)$ via Fokker-Planck equations.

    Different from other regular approaches to analyzing SGLD and related algorithms \cite{jin2020random,zou2021faster,raginsky2017non}, our method begins with continuous time evolution of the distributions $\bar{\rho}_t$, namely, the Fokker-Planck equation derived in Lemma \ref{lmm:ulafp}. Then, with the continuous time evolution, we are able to calculate and estimate the time derivative:
    \begin{equation*}
    \begin{aligned}
    \frac{d}{dt} D_{KL}(\bar{\rho}_t\|\rho_t) &\leq - A_0  D_{KL}(\bar{\rho}_t\|\rho_t) 
    + \beta\, \mathbb{E}_{\xi_k} \left[ \int |\bar{b}_t^{\xi_k} - b^{\xi_k}|^2 \bar{\rho}_t^{\xi_k} dx \right]\\
    &+ \beta\, \mathbb{E}_{\xi_k,\tilde{\xi}_k} \left[\int |b^{\xi_k}-b|^2 \frac{|\bar{\rho}_t^{\tilde{\xi}_k} - \bar{\rho}_t^{\xi_k}|^2}{\bar{\rho}_t^{\tilde{\xi}_k}} dx\right] := - A_0  D_{KL}(\bar{\rho}_t\|\rho_t)  + (I) + (II).
    \end{aligned}
    \end{equation*}
    Above, the densities $\bar{\rho}^{\xi_k}$ and drifts $b^{\xi_k}$ correspond to the fixed-batch version of SGLD, the function $\bar{b}_t^{\xi_k}$ is of the form of backward conditional expectation as we derived in Lemma \ref{lmm:ulafp}, and $\xi_k$, $\tilde{\xi}_k$ are independent. More details and rigorous definitions would be given at the beginning of Section \ref{sec:detailedproof}. Moreover, in this step, we use the log-Sobolev inequality in Assumption \ref{ass:pi}, which is essential to make our error bound uniform-in-time.

    \item Estimate the error term $(I)$ arising from time discretization via Taylor's expansion and integration by parts.

    At first glance, after Taylor's expansion for $\bar{b}_t^{\xi_k} - b^{\xi_k}$, one can only obtain a total $O(\eta)$ bound if a naive estimation via Jensen's inequality and tower property of the conditional expectations is conducted. However, since the function $\bar{b}_t^{\xi_k}$ is a backward conditional expectation, it is possible to cancel out the error term brought by the Brownian motion via more careful, complicated calculus. In fact, applying Bayes' law and integration by parts to the leasing term in the Taylor's expansion, and the second-order smoothness assumption in Assumption \ref{ass:b} to the remainder term in the Taylor's expansion, we have the following estimate for $(I)$:
    \begin{equation*}
        (I) \lesssim \eta^2(1 + \mathcal{I}(\bar{\rho}_{T_k}))
    \end{equation*}

    \item Estimate the error term $(II)$ arising from the random batch via techniques including Girsanov's theorem and some basic properties of path measures.

    From the boundedness assumption in \ref{ass:b}, it suffices to estimate $\int \left|\frac{\bar{\rho}_t^{\xi_k}}{\bar{\rho}_t^{\tilde{\xi}_k}}-1\right|^2 \bar{\rho}_t^{\tilde{\xi}_k} dx$. Note that $\bar{\rho}_t^{\xi_k}$, $\rho_t^{\xi_k}$ are time marginal distributions of solutions to two SDEs with same initial, same diffusion, but different drifts. Therefore, one may obtain an explicit expression using the following property of path measures (More details can be found near \eqref{pathmeasure} and Appendix \ref{sec:pathmeasure})
    \begin{equation*}
    \frac{\bar{\rho}_t^{\xi_k}}{\bar{\rho}_t^{\tilde{\xi}_k}} (x) = \mathbb{E} \left[\frac{dP_{\bar{X}}}{dP_{\bar{X}'}} \Bigg| \bar{X}'_t = x, \xi_k, \tilde{\xi}_k  \right], \quad t \in [T_k,T_{k+1})
    \end{equation*}
    where $P_{\bar{X}}$, $P_{\bar{X}'}$ are laws of $\bar{X}$, $\bar{X}'$ in the path space in $C\left([T_k,t]; \mathbb{R}^d \right)$ and note that $\bar{X}$, $\bar{X}'$ are two copies of SGLD with random batches $\xi_k$, $\tilde{\xi}_k$. Hence, after Girsanov's transform, the error term $(II)$ turns into a backward conditional expectation of some exponential minus one (see the expression after \eqref{eq:tildeb}). Finally, after similar calculus for $(I)$, we have the following estimate for $(II)$:
    \begin{equation*}
        (II) \lesssim \eta^2(1 + \mathcal{I}(\bar{\rho}_{T_k})).
    \end{equation*}
\end{enumerate}

\subsection{Proof of Proposition \ref{local_estimate}}\label{sec:detailedproof}
\begin{proof}[Proof of Proposition \ref{local_estimate}:]
We prove this local result following three main steps.

\textbf{STEP 1:} estimate time derivative of KL-divergence $\frac{d}{dt} D_{KL}(\bar{\rho}_t\|\rho_t)$ via Fokker-Planck equations.

Firstly, note that SGLD at discrete time points is a Markov chain (which is time-homogeneous when $\eta_k \equiv \eta$ is a constant), and $\bar{\rho}_{T_k}$  is the law at $T_k$. Recall that $\rho_t^{\b{\xi}}$ is the probability density of the fixed-batch version of SGLD \eqref{sgld_continuous} for a given sequence of batches $\b{\xi}:=(\xi_0, \xi_1, \cdots, \xi_k, \cdots)$ so that $\bar{\rho}_{T_k} = \mathbb{E}_{\xi}\left[\rho_{T_k}^{\b{\xi}}\right]$. Moreover, by Markov property, we are able to define
\begin{equation}\label{eq:Lstar}
    \bar{\rho}_t^{\xi_k} := \mathbb{E}\left[\rho_t^{\b{\xi}} \Big| \xi_i,\,i\geq k\right] = \cS_{T_k,t}^{\xi_k} \bar{\rho}_{T_k}, \quad t \in [T_k,T_{k+1}),
\end{equation}
where the operator $\cS_{T_k,t}^{\xi_k}$ is the evolution operator from $T_k$ to $t$ for the Fokker-Planck equation of the continuous SGLD \eqref{sgld_continuous} derived in Lemma \ref{lmm:ulafp}, for some given $\xi_k$:  
\begin{equation}\label{FP_xi}
    \partial_t \bar{\rho}_t^{\xi_k} = -\nabla \cdot (\bar{\rho}_t^{\xi_k} \bar{b}_t^{\xi_k}) + \beta^{-1} \Delta \bar{\rho}_t^{\xi_k},\quad \bar{\rho}_{T_k}^{\xi_k}
    =\bar{\rho}_{T_k},
\end{equation}
where $t \in [T_k,T_{k+1})$ and
\begin{equation}
    \bar{b}^{\xi_k}_t(x) := \mathbb{E}\left[b^{\xi_k}\left( \bar{X}_{T_k}\right) |\bar{X}_t = x, \xi_k\right], \quad t \in \left[T_k,T_{k+1}\right).
\end{equation}

Next, we calculate $\frac{d}{dt} D_{KL}(\bar{\rho}_t\|\rho_t)$ based on \eqref{FP_xi}.
Since $\bar{\rho}_t = \mathbb{E}_{\xi_k}[\bar{\rho}_t^{\xi_k}]$ for $t \in [T_k,T_{k+1})$, by \eqref{FP_xi} we have
\begin{equation}\label{FPbarpi}
\partial_t \bar{\rho}_t = \mathbb{E}_{\xi_k}
\left[-\nabla \cdot (\bar{b}_t^{\xi_k} \bar{\rho}_t^{\xi_k})\right] + \beta^{-1} \Delta \bar{\rho}_t.
\end{equation}
Then, using the Fokker-Planck equations \eqref{FPbarpi}, \eqref{FP} for $\bar{\rho}_t$ and $\rho_t$, respectively, we are able to calculate the following time derivative of the KL-divergence in the time interval $[T_k,T_{k+1})$:

\begin{equation}\label{dtKL}
\begin{aligned}
    &\quad\frac{d}{dt} D_{KL}(\bar{\rho}_t\|\rho_t)
    = \int \left( \partial_t \bar{\rho}_t \right) \left(\log \frac{\bar{\rho}_t}{\rho_t} + 1 \right)dx + \int \left( \partial_t \rho_t \right)\left(-\frac{\bar{\rho_t}}{\rho_t} \right) dx\\
   &=(\mathbb{E}_{\xi_k}(\bar{b}_t^{\xi_k} \bar{\rho}_t^{\xi_k})
   -\beta^{-1}\nabla\bar{\rho}_t)\cdot (\nabla \log \frac{\bar{\rho}_t}{\rho_t})+ (b\rho_t - \beta^{-1} \nabla \rho_t) \cdot (-\nabla \frac{\bar{\rho}_t}{\rho_t})  dx.\\
 &= \int \left((b\bar{\rho}_t - \beta^{-1}\nabla \bar{\rho}_t) \cdot (\nabla \log \frac{\bar{\rho}_t}{\rho_t}) + (b\rho_t - \beta^{-1} \nabla \rho_t) \cdot (-\nabla \frac{\bar{\rho}_t}{\rho_t}) \right) dx\\
 &   \quad + \int \mathbb{E}_{\xi_k}\left[(\bar{b}_t^{\xi_k} - b^{\xi_k}) \bar{\rho}_t^{\xi_k} \right] \cdot (\nabla \log \frac{\bar{\rho}_t}{\rho_t}) dx
    + \int \mathbb{E}_{\xi_k}\left[b^{\xi_k}\bar{\rho}_t^{\xi_k} - b\bar{\rho}_t \right] \cdot (\nabla \log \frac{\bar{\rho}_t}{\rho_t}) dx
\end{aligned}
\end{equation}
Note that $|b^{\xi_k}\bar{\rho}_t^{\xi_k} - b\bar{\rho}_t|$ is of order $O(1)$. The mechanism is that $\bar{\rho}_t^{\xi_k}$ is close to $\bar{\rho}_t$ and using the consistency of the random batch $\mathbb{E}_{\xi_k}\left[b^{\xi_k}(x)\right] = b(x)$. This fact can help cancel the leading error to some extent, by observing $\mathbb{E}_{\xi_k}[b^{\xi_k}\bar{\rho}_t^{\xi_k}-b\bar{\rho}_t]=\mathbb{E}_{\xi_k}[(b^{\xi_k}-b)(\bar{\rho}_t^{\xi_k}-\bar{\rho}_t)]$ and clearly the right-hand side is a small term. Our detailed strategy here is to rearrange the terms to get
\begin{equation}\label{eq4_9}
    \begin{aligned}
  \frac{d}{dt} D_{KL}(\bar{\rho}_t\|\rho_t)
    & = \left(-\beta^{-1} \int |\nabla \log \frac{\bar{\rho}_t}{\rho_t}|^2  \bar{\rho_t} dx\right)
    + \left(\int \mathbb{E}_{\xi_k}\left[(\bar{b}_t^{\xi_k} - b^{\xi_k}) \bar{\rho}_t^{\xi_k} \right] \cdot (\nabla \log \frac{\bar{\rho}_t}{\rho_t}) dx \right)\\
    & \quad + \left( \int \mathbb{E}_{\xi_k}\left[(b^{\xi_k} - b)(\bar{\rho}_t^{\xi_k} - \bar{\rho}_t) \right] \cdot (\nabla \log \frac{\bar{\rho}_t}{\rho_t}) dx \right)\\
    & := J_1 + J_2 + J_3.
    \end{aligned}
\end{equation}


For $J_2$, by Young's inequality, 
\begin{equation}\label{roma2}
\begin{aligned}
     J_2 & =\mathbb{E}_{\xi_k}\left[ \int (\bar{b}_t^{\xi_k} - b^{\xi_k}) \bar{\rho}_t^{\xi_k}  \cdot (\nabla \log \frac{\bar{\rho}_t}{\rho_t}) dx\right]\\
     & \leq \beta\mathbb{E}_{\xi_k}\int |\bar{b}_t^{\xi_k} - b^{\xi_k}|^2 \bar{\rho}_t^{\xi_k} dx + \frac{1}{4\beta} \int |\nabla \log \frac{\bar{\rho}_t}{\rho_t}|^2  \bar{\rho_t} dx ,
\end{aligned}
\end{equation}
where $\gamma$ is a positive constant.

For $J_3$, our approach is to introduce another copy of SGLD $\bar{X}'$ such that 
\begin{itemize}
\item $\bar{X}'_{T_k} = \bar{X}_{T_k}$;
\item  the Browmian Motion  are the same in $[T_k,T_{k+1})$;
\item the batch $\tilde{\xi}_k$ on $[T_k,T_{k+1})$ is independent of $\xi_k$.
\end{itemize}
 Consequently, the corresponding density of the law $\bar{\rho_t}^{\tilde{\xi}_k}$ for $\bar{X}'$ satisfy both  \eqref{eq:Lstar} and \eqref{FP_xi}, with $\bar{\rho}_{T_k}^{\tilde{\xi}_k} = \bar{\rho}_{T_k}$. Using the consistency of the random batch, we are able to obtain
\begin{equation}
    J_3 = \mathbb{E}_{\xi_k, \tilde{\xi}_k} \int \left[(b^{\xi_k} - b)(\bar{\rho}_t^{\xi_k} - \bar{\rho}_t^{\tilde{\xi}_k}) \right] \cdot (\nabla \log \frac{\bar{\rho}_t}{\rho_t}) dx
\end{equation}
Then by Young's inequality we have:
\begin{equation}\label{eq4_12}
    J_3 \leq \beta \mathbb{E}_{\xi_k,\tilde{\xi}_k}  \left[ \int |b^{\xi_k}-b|^2\frac{|\bar{\rho}_t^{\tilde{\xi}_k} - \bar{\rho}_t^{\xi_k}|^2}{\bar{\rho}_t^{\tilde{\xi}_k}} dx \right] + \frac{1}{4\beta} \int |\nabla \log \frac{\bar{\rho}_t}{\rho_t}|^2  \bar{\rho_t} dx
\end{equation}
where we applied the fact $\mathbb{E}_{\xi_k,\tilde{\xi}_k} |\nabla \log \frac{\bar{\rho}_t}{\rho_t}|^2  \bar{\rho}_t^{\tilde{\xi}_k} =|\nabla \log \frac{\bar{\rho}_t}{\rho_t}|^2  \bar{\rho}_t$. The introduction of the independent copy of $\tilde{\xi}_k$ is useful since we may apply the Girsanov transform later to estimate this quantitatively.

Now by Proposition \ref{LSIpropagation}, $\rho_t$ satisfies a LSI with a constant $\lambda^2\kappa(\beta)$, namely, $\forall f \in C_{>0}^{\infty}$,
\begin{equation}
    Ent_{\rho_t} (f) := \int f \log f d\rho_t -\left(\int f d\rho_t\right) \log\left(\int f d\rho_t\right) \leq \lambda^2 \kappa(\beta) \int \frac{|\nabla f|^2}{f} d\rho_t.
\end{equation}
Taking $f = \frac{\bar{\rho}_t}{\rho_t}$, one has
\begin{equation}\label{LSIapply}
    \int |\nabla \log \frac{\bar{\rho}_t}{\rho_t}|^2  \bar{\rho_t} dx \geq \frac{1}{\lambda^2 \kappa(\beta)} D_{KL}(\bar{\rho}_t\|\rho_t).
\end{equation}
Then, combining \eqref{eq4_9}, \eqref{roma2}, \eqref{eq4_12}, and \eqref{LSIapply}, one has 
\begin{equation}\label{314}
\begin{aligned}
    \frac{d}{dt} D_{KL}(\bar{\rho}_t\|\rho_t) &\leq - A_0  D_{KL}(\bar{\rho}_t\|\rho_t) \\
    &+ \beta\, \mathbb{E}_{\xi_k} \left[ \int |\bar{b}_t^{\xi_k} - b^{\xi_k}|^2 \bar{\rho}_t^{\xi_k} dx \right] + \beta\, \mathbb{E}_{\xi_k,\tilde{\xi}_k} \left[\int |b^{\xi_k}-b|^2 \frac{|\bar{\rho}_t^{\tilde{\xi}_k} - \bar{\rho}_t^{\xi_k}|^2}{\bar{\rho}_t^{\tilde{\xi}_k}} dx\right],
\end{aligned}
\end{equation}
where
\begin{equation}
    A_0 = \frac{\beta^{-1}}{2\lambda^2 \kappa(\beta)}.
\end{equation}
is a $\beta$-free constant. With \eqref{314}, to obtain the final estimate, one then needs to obtain an $O(\eta_k^2)$ estimate for the two terms:  $\mathbb{E}_{\xi_k} \mathbb{E}\left[ |\bar{b}^{\xi_k}_t(\bar{X}_t) - b^{\xi_k}(\bar{X}_t)|^2\Big| \xi_k\right]$ and $\mathbb{E}_{\xi_k,\tilde{\xi}_k} \left[\int |b^{\xi_k}-b|^2 \frac{|\bar{\rho}_t^{\tilde{\xi}_k} - \bar{\rho}_t^{\xi_k}|^2}{\bar{\rho}_t^{\tilde{\xi}_k}} dx\right]$ .

\textbf{STEP 2:} estimate $\mathbb{E}_{\xi_k} \mathbb{E}\left[ |\bar{b}^{\xi_k}_t(\bar{X}_t) - b^{\xi_k}(\bar{X}_t)|^2\Big| \xi_k\right]$

In this step, we aim to obtain an $O(\eta_k^2)$ bound for $\mathbb{E}_{\xi_k} \mathbb{E}\left[ |\bar{b}^{\xi_k}_t(\bar{X}_t) - b^{\xi_k}(\bar{X}_t)|^2\Big| \xi_k\right]$ based on Taylor expansion for $b^{\xi_k}$. Note that $\xi_k$ is fixed so we may follow the estimate for the Unadjusted Langevin Algorithm (ULA) in \cite{mou2019improved} in this step.

By Taylor expansion, $\forall t \in [T_k,T_{k+1})$,
\begin{equation*}
    \begin{aligned}
    &\quad \bar{b}^{\xi_k}_t(x) - b^{\xi_k}(x)  = \mathbb{E} \left[b^{\xi_k}(\bar{X}_{T_k})- b^{\xi_k}(\bar{X}_t) | \bar{X}_t = x, \xi_k \right]\\
    &= \mathbb{E} [\bar{X}_{T_k} - \bar{X}_t | \bar{X}_t = x, \xi_k ]\cdot \nabla b^{\xi_k}(x)\\
    & \quad +\frac{1}{2} \mathbb{E} \left[\left(\bar{X}_{T_k}- \bar{X}_t\right)^{\otimes 2} : \int_0^1  \nabla^2 b^{\xi_k} \left((1-s)\bar{X}_t + s\bar{X}_{T_k} \right)ds  \Big| \bar{X}_t = x , \xi_k\right].
    \end{aligned}
\end{equation*}
For simplicity, we denote 
\begin{equation}\label{defrt}
    \bar{r}_t(x) := \frac{1}{2}\mathbb{E} \left[\left(\bar{X}_{T_k} - \bar{X}_t\right)^{\otimes 2} : \int_0^1  \nabla^2 b^{\xi_k} \left((1-s)\bar{X}_t + s\bar{X}_{T_k} \right)ds  \Big| \bar{X}_t = x, \xi_k\right].
\end{equation}
In Lemma \ref{lmmrt} , we show that for $t \in \left[T_k, T_{k+1} \right)$,
\begin{equation}
    \mathbb{E}\left[ |\bar{r}_t(\bar{X}_t)|^2 \Big| \xi_k \right]\leq c d \beta^{-2}(t-T_k)^2,
\end{equation}
where $c$ is a positive constant independent of $k$, $d$ and $\xi_k$.
For the first term $  \mathbb{E} [\bar{X}_{T_k} - \bar{X}_t | \bar{X}_t = x, \xi_k]\cdot \nabla b^{\xi_k}(x)$, by Bayes' law we have
\begin{equation}\label{319}
\begin{aligned}
     \mathbb{E} [\bar{X}_{T_k} - \bar{X}_t| \bar{X}_t^{\xi}  = x , \xi_k ] &= \int (y-x)P(\bar{X}_{T_k} = y | \bar{X}_t = x,\xi_k) dy\\
     & = \int (y-x) \frac{\bar{\rho}_{T_k}^{\xi_k}(y)P(\bar{X}_t = x | \bar{X}_{T_k} = y,\xi_k)}{\bar{\rho}_t^{\xi_k}(x)} dy.
\end{aligned}
\end{equation}
Clearly, the distribution $P(\bar{X}_t = x | \bar{X}_{T_k} = y,\xi_k)$ is Gaussian, namely,
\begin{equation}
    P\left(\bar{X}_t=x|\bar{X}_{T_k}=y,\xi_k\right) = \left(4\pi \beta^{-1} (t-T_k)\right)^{-\frac{d}{2}}\exp\left(-\frac{|x-y-b^{\xi_k}(y)(t-T_k)|^2}{4\beta^{-1}(t-T_k)}\right),
\end{equation}
which motivates us to calculate \eqref{319} via integration by parts. Indeed, we show in Lemma \ref{integralbyparts} that
\begin{equation}\label{320}
    \int \left|\mathbb{E} \left[\bar{X}_{T_k} - \bar{X}_t | \bar{X}_t  = x, \xi_k\right]\right|^2 \bar{\rho}_t^{\xi_k}(x)dx \leq c \eta_k^2 \left(d(1 + \beta^{-1})+\beta^{-2}\mathcal{I}(\bar{\rho}_{T_k})\right),\quad \forall t\in [T_k,T_{k+1}), 
\end{equation}
where $\mathcal{I}(\rho) := \int \rho |\nabla \log \rho|^2 dx$ is the Fisher information, and  $c$ is a positive constant independent of $k$, $d$ and $\xi_k$. Then, it is left to give an $O(1)$ estimate for each Fisher information $\mathcal{I}(\bar{\rho}_{T_k})$ in ULA, which has been done in Section \ref{auxiliary}.

\textbf{STEP 3:} estimate $\mathbb{E}_{\xi_k,\tilde{\xi}_k} \left[\int |b^{\xi_k}-b|^2 \frac{|\bar{\rho}_t^{\tilde{\xi}_k} - \bar{\rho}_t^{\xi_k}|^2}{\bar{\rho}_t^{\tilde{\xi}_k}} dx\right]$.

By the boundedness assumption in Assumption \ref{ass:b}, we only need to estimate $\int  \frac{|\bar{\rho}_t^{\tilde{\xi}_k} - \bar{\rho}_t^{\xi_k}|^2}{\bar{\rho}_t^{\tilde{\xi}_k}} dx $. Recall that $\bar{\rho}_t^{\xi_k}$ and $\bar{\rho}_t^{\tilde{\xi}_k}$ are the densities of the laws of $\bar{X}$ and $\bar{X}'$ with two independent batches $\xi_k$ and $\tilde{\xi}_k$ for interval $[T_k, T_{k+1})$. Hence the problem is again reduced to the ULA case.

We first note that
\begin{equation}
    \int  \frac{|\bar{\rho}_t^{\tilde{\xi}_k} - \bar{\rho}_t^{\xi_k}|^2}{\bar{\rho}_t^{\tilde{\xi}_k}} dx = \int \left|\frac{\bar{\rho}_t^{\xi_k}}{\bar{\rho}_t^{\tilde{\xi}_k}}-1\right|^2 \bar{\rho}_t^{\tilde{\xi}_k} dx,
\end{equation}
and  the property of path measures gives
\begin{equation}\label{pathmeasure}
    \frac{\bar{\rho}_t^{\xi_k}}{\bar{\rho}_t^{\tilde{\xi}_k}} (x) = \mathbb{E} \left[\frac{dP_{\bar{X}}}{dP_{\bar{X}'}} \Bigg| \bar{X}'_t = x, \xi_k, \tilde{\xi}_k  \right],
\end{equation}
where $P_{\bar{X}}$ and $P_{\bar{X}'}$ are path measures in $C\left([T_k,T_{k+1}]; \mathbb{R}^d \right)$. 
Let
\[
\bar{X}_{T_k}=\bar{X}'_{T_k}=y\sim \bar{\rho}_{T_k}.
\]
Then for $t \in [T_k,T_{k+1})$, Girsanov transform gives:
\begin{equation}\label{eq:gir}
    \frac{dP_{\bar{X}}}{dP_{\bar{X}'}} (\bar{X}') = \exp\left(\sqrt{\frac{\beta}{2}}\int_{T_k}^{T_{k+1}} (b^{\tilde{\xi}_k} - b^{\xi_k}) (y) dW_s - \frac{\beta}{4} \int_{T_k}^{T_{k+1}} \left|(b^{\tilde{\xi}_k} - b^{\xi_k}) (y)\right|^2 ds\right).
\end{equation}
Details for \eqref{pathmeasure} and \eqref{eq:gir} can be found in Corollary \ref{coro:pathmeasure} of Appendix \ref{sgldpf}. 
Denote
\begin{equation}\label{eq:tildeb}
    \tilde{b}(y) := \left(b^{\tilde{\xi}_k} - b^{\xi_k}\right)(y).
\end{equation}
Then, for $t \in [T_k, T_{k+1})$, the martingale property gives
\[
\frac{\bar{\rho}_t^{\xi_k}}{\bar{\rho}_t^{\tilde{\xi}_k}} (x)
=\mathbb{E} \left[
\exp\left(\sqrt{\frac{\beta}{2}}\int_{T_k}^t \tilde{b} (y) dW_s - \frac{\beta}{4} \int_{T_k}^t |\tilde{b} (y)|^2 ds\right) \Bigg| \bar{X}'_t = x, \xi_k, \tilde{\xi}_k  \right].
\]
Hence,
\begin{equation*}
\begin{aligned}
    &\quad \int \left|\frac{\bar{\rho}_t^{\xi_k}}{\bar{\rho}_t^{\tilde{\xi}_k}}-1\right|^2 \bar{\rho}_t^{\tilde{\xi}_k} dx \\
    &= \int \bar{\rho}_t^{\tilde{\xi}_k}(x)\left|\mathbb{E}\left[e^{\sqrt{\frac{\beta}{2}}\int_{T_k}^t \tilde{b} (\bar{X}'_{T_k}) dW_s - \frac{\beta}{4} \int_{T_k}^t |\tilde{b} (\bar{X}'_{T_k})|^2 ds} \Big| \bar{X}'_t = x, \xi_k, \tilde{\xi}_k \right] - 1\right|^2 dx\\
    & = \int \bar{\rho}_t^{\tilde{\xi}_k} (x) \left|\int \left(  e^ {\frac{\beta}{2}\,\tilde{b}(y)(x-y-(t-T_k)b^{\tilde{\xi}_k}(y)) - \frac{\beta}{4} |\tilde{b}(y)|^2(t-T_k)} - 1 \right) P(\bar{X}_{T_k} = y|\bar{X}'_t = x,\xi_k,\tilde{\xi}_k) dy \right|^2 dx\\
\end{aligned}
\end{equation*}
Above we have used the fact that $\bar{X}'_t=y+b^{\tilde{\xi}_k}(y)(t-T_k)+2\beta^{-1}(W_t-W_{T_k})$, resulting in that
\[
\sqrt{2\beta^{-1}}(W_t-W_{T_k})
=x-y-b^{\tilde{\xi}_k}(y)(t-T_k).
\]
Now in order to handle the integration of the term $e^ {\frac{\beta}{2}\tilde{b}(y)(x-y-(t-T_k)b^{\tilde{\xi}_k}(y)) - \frac{\beta}{4} |\tilde{b}(y)|^2(t-T_k)}$, we split it by
\begin{equation}
    \begin{aligned}
    & \quad e^ {\frac{\beta}{2}\,\tilde{b}(y)(x-y-(t-T_k)b^{\tilde{\xi}_k}(y)) - \frac{\beta}{4} |\tilde{b}(y)|^2(t-T_k)} - 1 \\
    &=  \frac{\beta}{2}\,\tilde{b}(y)(x-y) + \left(-\frac{\beta}{2}\,\tilde{b}(y) b^{\tilde{\xi}_k}(y)- \frac{\beta}{4} |\tilde{b}(y)|^2\right)(t-T_k)  + \big(e^z-z-1 \big)  \\
    & := K_1 + K_2 + K_3,
    \end{aligned}
\end{equation}
where 
\begin{equation}
    z  := \frac{\beta}{2}\,\tilde{b}(y)(x-y-(t-T_k)b^{\tilde{\xi}_k}(y)) - \frac{\beta}{4} |\tilde{b}(y)|^2(t-T_k).
\end{equation}
Then
\begin{equation}
    \int \left|\frac{\bar{\rho}_t^{\xi_k}}{\bar{\rho}_t^{\tilde{\xi}_k}}-1\right|^2 \bar{\rho}_t^{\tilde{\xi}_k} dx = \int \bar{\rho}_t^{\tilde{\xi}_k} (x) \left|\int \left(K_1 + K_2 + K_3 \right) P(\bar{X}_{T_k} = y|\bar{X}'_t = x,\xi_k,\tilde{\xi}_k) dy \right|^2dx.
\end{equation}

For $K_1 = \frac{\beta}{2}\,\tilde{b}(y)(x-y)$, using integration by parts again, we prove in Lemma \ref{lmm:i} in Section \ref{sgldpf} that  for $t \in [T_k,T_{k+1})$,
\begin{equation}
    \int \bar{\rho}_t^{\tilde{\xi}_k} (x) \left(\int K_1 \,  P(\bar{X}_{T_k} = y|\bar{X}'_t = x,\xi_k,\tilde{\xi}_k) dy \right)^2dx \leq c \eta_k^2  \left(d(\beta^2+1)+\mathcal{I}(\bar{\rho}_{T_k})\right).
\end{equation}
where $c$ is a positive constant independent of $k$, $d$, $\tilde{\xi_k}$ and $\xi_k$.

For $K_2 = (-\frac{\beta}{2}\,\tilde{b}(y) \cdot b^{\tilde{\xi}_k}(y))- \frac{\beta}{4} |\tilde{b}(y)|^2)(t-T_k)$, using the boundedness and Lipschitz condition in Assumption \ref{ass:b}, we prove in Lemma \ref{lmm:ii} that  for $t \in [T_k,T_{k+1})$,
\begin{equation}
     \int \bar{\rho}_t^{\tilde{\xi}_k} (x) \left(\int K_2 \, P(\bar{X}_{T_k} = y|\bar{X}'_t = x,\xi_k,\tilde{\xi}_k) dy \right)^2dx \leq c d\beta(\beta+1)\eta_k^2,
\end{equation}
where $c$ is a positive constant independent of $k$, $d$, $\tilde{\xi}_k$ and $\xi_k$.

For the remaining term $K_3$, after applying Jensen's inequality and the tower property of conditional expectation, we prove in Lemma \ref{lmm:iii} that for $t \in [T_k,T_{k+1})$, 
\begin{equation}\label{iiibound}
     \int \bar{\rho}_t^{\tilde{\xi}_k} (x) \left(\int K_3 \, P(\bar{X}_{T_k} = y|\bar{X}'_t = x,\xi_k,\tilde{\xi}_k) dy \right)^2dx \leq c d\beta^2\eta_k^2.
\end{equation}
where $c$ is a positive constant independent of $k$, $d$, $\tilde{\xi}_k$ and $\xi_k$.


Equipped with the estimation for the integration of $K_1$, $K_2$, and $K_3$, one finally obtains
\begin{equation}\label{328}
    \int \left|\frac{\bar{\rho}_t^{\xi_k}}{\bar{\rho}_t^{\tilde{\xi}_k}}-1\right|^2 \bar{\rho}_t^{\tilde{\xi}_k} dx \leq c \eta_k^2 \left(d(\beta^2+1) + \mathcal{I}(\bar{\rho}_{T_k})\right) , \quad \forall t \in [T_k,T_{k+1}),
\end{equation}
where $c$ is a positive constant independent of $k$, $d$, $\tilde{\xi}_k$ and $\xi_k$. Combining with the estimate for Fisher information obtained in Section \ref{auxiliary}, one is able to get an $\eta_k^2$ estimation for the term $\mathbb{E}_{\xi_k,\tilde{\xi}_k} \left[\int |b^{\xi_k}-b|^2 \frac{|\bar{\rho}_t^{\tilde{\xi}_k} - \bar{\rho}_t^{\xi_k}|^2}{\bar{\rho}_t^{\tilde{\xi}_k}} dx\right]$. Note that after taking the expectation $\mathbb{E}_{\xi_k,\tilde{\xi}_k}[\cdot]$, the bound is still of order $\eta_k^2$, since the bounds above are all independent of the batches $\xi_k$, $\tilde{\xi}_k$.

Combining the results in STEP 2 and STEP 3 finally leads to the desired result:
\begin{multline}
    \frac{d}{dt} D_{KL}(\bar{\rho}_t\|\rho_t) \leq - A_0 D_{KL}(\bar{\rho}_t\|\rho_t) \\
    + A_1 \eta_k^2 \left(d (\beta^3+\beta^{-1}) +  (\beta + \beta^{-1}) \mathcal{I}(\bar{\rho}_{T_k})\right),\quad \forall t \in [T_k,T_{k+1}),
\end{multline}
where $A_0$, $A_1$ are positive constants independent of $k$, $d$, $\beta$ and $\xi_k$.
\end{proof}

\section{Discussion}\label{sec:discussion}
\subsection{Convergence to equilibrium and target distribution}
There are two questions researchers in related fields may be interested in: whether the SGLD algorithm can converge to the target distribution? whether the SGLD algorithm itself has an invariant measure and whether can it converge to its only equilibrium? In the following, we answer both questions, respectively.
\begin{itemize}
\item \textbf{Convergence to the target distribution}
    
Viewing SGLD as a popular sampling algorithm, many researchers would care about the convergence rate to the target distribution. An important extension of Theorem \ref{longtimesgld_etak} or Corollary \ref{longtimesgld} is about the distance between the SGLD iteration and the target distribution, namely, the invariant distribution of the Langevin diffusion \eqref{eq:overdampedlangevin}. When we are doing sampling tasks, this result is an important measure for the efficiency of a sampling algorithm. Note that the rate in our result can be viewed as a great improvement compared with former work like \cite{mou2018generalization,zou2021faster,farghly2021time}, where the optimal bound so far is of order $O(\sqrt{\eta})$ in terms of Wasserstein distance or total variation. Meanwhile, as will be derived in the following, we obtain a bound of order $O(\eta)$.

To compare our result with former ones more directly, we consider the constant step size (or learning rate) $\eta$. We consider the Wasserstein-2 ($W_2$) distance, the Wasserstein-1 ($W_1$) distance, and the total variation (TV) distance in the following discussion. Firstly let us  recall the definition of  $W_p$  distance ($p=1,2$ here) and TV distance between two distributions $\mu$ and $\nu$ \cite{santambrogio2015optimal}:
\begin{equation}
    W_{p}(\mu, \nu):=\left(\inf _{\gamma \in \Pi(\mu, v)} \int_{\mathbb{R}^{d} \times \mathbb{R}^{d}}|x-y|^{p} d \gamma\right)^{1 / p},
\end{equation}
and
\begin{equation}
    D_{TV}(\mu,\nu) := \sup_{A\in \mathcal{B}(\mathbb{R}^d)} |\mu(A) - \nu(A)|. 
\end{equation}
Here, $\Pi(\mu, \nu)$ means all joint distributions whose marginal distributions are $\mu$ and $\nu$, respectively.

Now, by Corollary \ref{longtimesgld}, we have (recall $g(\beta)$ and $A_0$ defined in Theorem \ref{longtimesgld_etak})
\begin{equation}
    D_{KL}(\bar{\rho}_t\|\rho_t) \leq cd(\beta^2+1 + g(\beta)/A_0)\eta^2. \quad t>0,
\end{equation}
where $c$ is independent of $t$, and $\eta$ is the constant step size (or learning rate).

Also, since the invariant distribution $\pi$ satisfies a Log-Sobolev inequality by Assumption \ref{ass:pi}, we have the following exponential convergence for the Langevin diffusion \eqref{eq:overdampedlangevin}, which is a classical result \cite{markowich2000trend}:
\begin{equation}
    D_{KL}(\rho_t\|\pi) \leq e^{-\gamma t}D_{KL}(\rho_0\|\pi)
\end{equation}

It is well-known that we can bound the $W_2$, $W_1$, and TV distance with square root of the KL-divergence by Talagrand transportation inequality \cite{otto2000generalization,talagrand1991new}, the weighted Csiszar-Kullback-Pinsker inequality \cite{bolley2005weighted}, and the Pinsker's inequality \cite{pinsker1963information}, respectively. 
Note that for $W_2$ distance and TV distance, the constant $c_1'$ above is dimension-free; for $W_1$ distance, the constant $c_1'$ above is of $O(d^{\frac{1}{2}})$.
Denote $\bar{g}(\beta) := \beta^2 + 1 + g(\beta) / A_0$ (recall $g(\beta)$ and $A_0$ defined in Theorem \ref{longtimesgld_etak}).
Together with the triangle inequality for $W_2$, $W_1$, and TV distances,  one has the following:
\begin{corollary}\label{coro:equilibrium}
Consider the probability density functions $\rho_t$, $\bar{\rho}_t$ for $X_t$, $\bar{X}_t$ defined in \eqref{eq:overdampedlangevin}, \eqref{sgld_continuous}. Suppose Assumption \ref{ass:b}, \ref{ass:pi} hold. Then the following holds: 
\begin{itemize}
\item[(a)] If $\eta_k=\eta$, $\forall k$, then there exist positive constants  $c_1^{W_1}$, $c_2^{W_1}$, $c_1^{W_2}$, $c_2^{W_2}$, $c_1^{TV}$, $c_2^{TV}$, $\gamma$ , $\Delta_0$ independent of $t$, $\beta$ and $d$ such that for all $\eta \in (0, \Delta_0)$,
\begin{itemize}
    \item $W_1(\bar{\rho}_t, \pi) \leq c_1^{W_1} \bar{g}(\beta)^{\frac{1}{2}}d\eta + c_2^{W_1}d^{\frac{1}{2}} e^{-\frac{1}{2}\gamma t},$
    \item $W_2(\bar{\rho}_t, \pi) \leq c_1^{W_2}\bar{g}(\beta)^{\frac{1}{2}}d^{\frac{1}{2}}\eta + c_2^{W_2} e^{-\frac{1}{2}\gamma t},$
    \item $D_{TV}(\bar{\rho}_t, \pi) \leq c_1^{TV}\bar{g}(\beta)^{\frac{1}{2}}d^{\frac{1}{2}}\eta + c_2^{TV} e^{-\frac{1}{2}\gamma t}.$
\end{itemize}
\item[(b)] If $\eta_i = (\ell+i)^{-\theta}$, with $\theta\in (0,1)$ being the decay rate, then there exist positive constants $c^{W_1}$, $c^{W_2}$, $c^{TV}$ independent of $k$, $\beta$ and $d$ such that
\begin{itemize}
    \item $W_1(\bar{\rho}_{T_k}, \pi) \leq c^{W_1} g(\beta)^{\frac{1}{2}}d \left(\frac{1}{k}\right)^{\theta}$,
    \item $W_2(\bar{\rho}_{T_k}, \pi) \leq c^{W_2} g(\beta)^{\frac{1}{2}}d^{\frac{1}{2}} \left(\frac{1}{k}\right)^{\theta}$,
    \item $D_{TV}(\bar{\rho}_{T_k}, \pi) \leq c^{TV} g(\beta)^{\frac{1}{2}} d^{\frac{1}{2}} \left(\frac{1}{k}\right)^{\theta}$.
\end{itemize}
\end{itemize}
\end{corollary}

By Corollary \ref{coro:equilibrium}, we can conclude the steps of simulations for SGLD needed to achieve a tolerance $\epsilon$ under different distances.
\begin{corollary}
Under Assumption \ref{ass:b}, \ref{ass:pi}, to achieve an error smaller than $\epsilon$, one needs the following numbers of simulations of SGLD respectively  under the corresponding distances:
\begin{itemize}
    \item $N=O( \bar{g}(\beta)^{\frac{1}{2}}d\epsilon^{-1}(\log \epsilon^{-1} + \log d^{\frac{1}{2}}))$ for $W_1$ distance 
    (by setting $\eta = O(\bar{g}(\beta)^{-\frac{1}{2}}d^{-1} \epsilon)$);
   
    \item $N=O( \bar{g}(\beta)^{\frac{1}{2}}d^{\frac{1}{2}}\epsilon^{-1}\log \epsilon^{-1})$ for $W_2$ distance 
    (by setting $\eta = O(\bar{g}(\beta)^{-\frac{1}{2}}d^{-\frac{1}{2}} \epsilon)$) ;
    \item $N=O(\bar{g}(\beta)^{\frac{1}{2}}d^{\frac{1}{2}}\epsilon^{-1}\log \epsilon^{-1})$ for TV distance 
    (by setting $\eta = O(\bar{g}(\beta)^{-\frac{1}{2}}d^{-1} \epsilon)$).
\end{itemize}
\end{corollary}
These results are improved compared to existing results in literature.

\item \textbf{Ergodicity of SGLD}

An important property of SGLD that remains to be studied is its ergodicity, which then ensures the existence of the invariant distribution of the SGLD dynamic and enables us to explore the convergence of the algorithm itself. 
In \cite{brosse2018promises}, the authors considered the ergodicity of SGLD  under $W_2$ distance with the assumptions of global strong convexity of $U$ and 4-th order Lipschitz condition.

In our setup, we do not have the global strong convexity of $U$ and only have second order Lipschitz condition. 
Luckily, the ergodicity could be studied using reflection coupling \cite{eberle2011reflection,jin2022ergodicity,majka2020nonasymptotic,li2023geometric}, under mild assumptions.
In particular, we have the following claim, which is the main result in \cite{li2023geometric}.

In fact, we need to assume the followings, which can be understood as the positive upper and lower bound of $\nabla^2 U$ outside the circle $B(0,R)$ in $\mathbb{R}^d$.
\begin{assumption}\label{ass:reflection}
There exist $R_0 >0$, $\kappa_0 > 0$, $K > 0$ such that the followings hold:
    \begin{itemize}
        \item[(a)](locally nonconvex) The Hessian matrix of $U$ is uniformly positive definite outside $B(0,R_0)$, namely,
        \begin{equation}
            \nabla^2 U(x) \succeq \kappa_0 I_d, \quad \forall x \in \mathbb{R}^d \setminus B(0,R_0);
        \end{equation}
        \item[(b)](global uniform-in-batch Lipschitz) $\forall x,y \in \mathbb{R}^d$, $\forall \xi$,
        \begin{equation}
            \left|\nabla U^{\xi}(x) - \nabla U^{\xi}(y) \right| \leq K |x-y|.
        \end{equation}
    \end{itemize}
\end{assumption}
Then the following claim holds.
\begin{proposition}
Suppose Assumption \ref{ass:reflection} holds. Denote $\Delta_0 := \sup_k \eta_k$. There exist positive constants $\bar{\Delta}$, $c_0$, $c_1$ such that if $\Delta_0 < \bar{\Delta}$ , then for any two initial distributions $\mu_0$ and $\nu_0$, denoting $\bar{\mu}_t$ and $\bar{\nu}_t$ to be the corresponding time marginal distributions for the time continuous interpolation of SGLD algorithm \eqref{sgld_continuous}, it holds that 
\begin{equation}\label{W1contraction}
    W_1 (\bar{\mu}_{T_k}, \bar{\nu}_{T_k}) \leq c_0 e^{-c_1T_k} W_1(\mu_0,\nu_0).
\end{equation}
\end{proposition}

Consequently, if the step size (or learning rate) is constant $\eta_k\equiv \eta$ such that the discrete chain is time-homogeneous, then the SGLD as a discrete time Markov chain has an invariant measure $\tilde{\pi}$ by the Banach contraction mapping theorem. Then, we have the following conclusion
\begin{corollary}
Let $\eta_k\equiv \eta$. Denote by $\bar{\rho}_t$ the time marginal distribution of SGLD \eqref{sgld_continuous} at time $t$. Suppose Assumption \ref{ass:reflection} holds. Then there exist positive constants $\bar{\Delta}$, $c_0$, $c_1$ such that if $\eta \leq \bar{\Delta}$, then for any initial distribution $\bar{\rho_0}$ with finite first moment, the SGLD iteration has an invariant measure $\tilde{\pi}$. Moreover, $\bar{\rho}_t$ convergences exponentially to $\tilde{\pi}$ in terms of W-1 distance:
\begin{equation}\label{W1ergodicity}
    W_1(\bar{\rho}_{n\eta}, \tilde{\pi}) \leq c_0 e^{-cn\eta}W_1(\bar{\rho_0}, \tilde{\pi}).
\end{equation}
\end{corollary}

Equipped with the ergodicity results for SGLD and combining the results in Corollary \ref{coro:equilibrium}, we are able to estimate the distance between the target distribution and the invariant distribution of SGLD itself:
\begin{equation}
    W_1 (\tilde{\pi},\pi) \leq c\eta.
\end{equation}
This is improved compared to existing results in literature. Detailed derivation can be found in \cite{li2023geometric}.
\end{itemize}

\subsection{Other discussions}\label{subsec:moredis}
\begin{itemize}
\item \textbf{Justification of sharpness}

A natural question for our results is: is this bound we obtained really a ``sharp" one? To answer this question, we simply consider the following Ornstein-Uhlenbeck (O-U) process in $\mathbb{R}^1$, which satisfies all our assumptions.
\begin{equation}\label{OUprocess}
    dX = -Xdt + dW.
\end{equation}
Its solution has an explicit expression: 
\begin{equation}\label{eq_EM}
   X_t = e^{-t}X_0 + \int_0^t e^{s-t}dW_s. 
\end{equation}
Consider the full-batch SGLD, namely, the Euler-Maruyama scheme for \eqref{OUprocess} with constant step size $\eta$:
\begin{equation}
    \hat{X}_{T_{k+1}} = \hat{X}_{T_k} - \eta \hat{X}_{T_k} + (W_{T_{k+1}} - W_{T_k}).
\end{equation}
Suppose $X_t$, $\hat{X}_t$ are Gaussion, and $\mathbb{E}[X_0]\neq 0$, then we are able to calculate the KL-divergence without much difficulty. Indeed, by definition, at $T := k\eta$, their mean and variance are given by
\begin{equation*}
    \mu_1:=\mathbb{E}\left[X_{T}\right] = e^{-T}\mathbb{E}\left[X_0\right],
\end{equation*}
\begin{equation*}
    \sigma_1^2 := \mathrm{Var}(X_{T}) = e^{-2T}\mathrm{Var}(X_0) + \frac{1}{2}\left(1 - e^{-2T} \right),
\end{equation*}
and
\begin{equation*}
    \mu_2:=\mathbb{E}\left[\hat{X}_{T}\right]=(1-\eta)^{(T/\eta)}\mathbb{E}\left[X_0\right],
\end{equation*}
\begin{equation*}
    \sigma_2^2:=\mathrm{Var}(\hat{X}_{T})=(1-\eta)^{2(T/\eta)}\mathrm{Var}(X_0) + \frac{1}{2-\eta}\left(1 - (1-\eta)^{2(T/\eta)} \right).
\end{equation*}
Clearly, for small $\eta$,
\begin{equation*}
\begin{aligned}
    |\mu_1-\mu_2| &= \left|\mathbb{E}\left[X_0\right]\right|\left(e^{-T}-(1-\eta)^{(T/\eta)}\right) \ge \left|\mathbb{E}\left[X_0\right]\right|\left(e^{-t}-e^{-t-t\eta/2}\right)\\
    & >\left|\mathbb{E}\left[X_0\right]\right|e^{-T}\left(\frac{1}{2}\eta T-\frac{1}{8}\eta^2 T^2\right)>\frac{1}{4}\left|\mathbb{E}\left[X_0\right]\right|e^{-T} T\eta
\end{aligned}
\end{equation*}
and $\sigma_2^2 <  2\sigma_1^2$.

Also by direct calculation, the KL-divergence between two Gaussian distributions $N(\mu_1,\sigma_1^2)$, $N(\mu_2,\sigma_2^2)$ is given by
\begin{equation}
\begin{aligned}
    D_{KL}\left(N(\mu_1,\sigma_1^2)\|N(\mu_2,\sigma_2^2)\right) &=\left(\frac{1}{2}\log \frac{\sigma_2^2}{\sigma_1^2} + \frac{1}{2}\left(\frac{\sigma_1^2}{\sigma_2^2} - 1\right)\right) + \frac{(\mu_1 - \mu_2)^2}{2\sigma_2^2} \\
    & \geq 0 + \frac{(\mu_1 - \mu_2)^2}{2\sigma_2^2}. 
\end{aligned}
\end{equation}
Therefore,
\begin{equation}
     D_{KL}(\bar{\rho}_{T}\|\rho_{T}) \geq \frac{(1/16)\left|\mathbb{E}\left[X_0\right] \right|^2 e^{-2T}T^2}{4\sigma_1(T)^2}\,\eta^2.
\end{equation}
This then indicates that our $\eta^2$ bound is a sharp one.

\item \textbf{The role of consistency of  the random batch technique}

As a variant of ULA, the key difference of SGLD is the introduction of the mini-batch technique, which reduces the computational cost. To understand the error and mechanism for the sharp error estimate, let us take $\eta_k \equiv \eta$. 

For each single step, the error of the drift is $O(1)$. 
The methods involving random batches (such as the SGD and the random batch methods) have a strong error 
like (see for example \cite{jin2020random})
\[
\sqrt{\E\|X-\bar{X}\|^2} \sim \sqrt{(e^{k\eta}-1)\eta}.
\]
Clearly, the strong error decays like $\sqrt{\eta}$, which is actually sharp. As mentioned in \cite{jin2020random},
the averaging effect in time ensures a convergence like "law of large numbers" in time so the convergence rate is $\sqrt{\eta}$. 
The strong error estimate can imply that 
\[
(\E_{\xi} W_2^2(\rho_t^{\b{\xi}}, \rho_t))^{1/2} \sim \sqrt{(e^{k\eta}-1)\eta}.
\]
This indicates that the fluctuation of the trajectory and $\rho_t^{\b{\xi}}$ is really like $\sqrt{\eta}$. Hence, the existing error estimates of SGLD based on the trajectory estimates can achieve $\sqrt{\eta}$ at most. 
Moreover, such an estimate based on trajectory estimate can only give an $O(\eta)$ one step error of the SGLD under Wasserstein distance (the $O(\sqrt{\eta})$ global error is due to the time averaging over multiple intervals).

The most important property of the random batch is its consistency:
\begin{equation}\label{eq:consistencyrbm}
    \mathbb{E}_{\xi}\left[b^{\xi}(x) \right] = b(x),
\end{equation}
which can be interpreted as an unbiased approximation of the original drift function, as has been mentioned in \eqref{eq:unbiasedproperty}. 
Consider one step. Starting from a common $\rho_0$, after one step, formally each $\rho_{\eta}^{\b{\xi}}$ has the following expression
\[
\rho_{\eta}=\rho_0+\eta \mathcal{L}_{\xi}^*\rho_0+O(\eta^2),
\]
where $\mathcal{L}_{\xi}$ is the generator corresponding to batch $\xi$, while
\[
\rho_{\eta}=\rho_0+\eta\mathcal{L}^*\rho_0+O(\eta^2).
\]
Hence, the error is like $O(\eta)$. Since the error of the drift can cancel if one takes average across batches and thus one is motivated to take the average of all possible $\rho_t^{\b{\xi}}$:
\begin{equation}
    \bar{\rho}_t = \mathbb{E}_{\xi} \left[\rho_t^{\b{\xi}}\right],
\end{equation}
which is the true law of the SGLD. The obtained $\bar{\rho}_t$ is then expected to have $O(\eta^2)$ local error compared to $\rho_{t_{n+1}}$. 
Of course, a direct proof using Wasserstein distances is difficult, so we utilize the KL divergence to accomplish the proof, motivated by the recent work \cite{mou2019improved}. Anyway,  this intuition forms the foundation of our proof and is reflected in how we treat the terms in \eqref{dtKL}.In fact, as discussed after \eqref{dtKL}, the key step is to rearrange
\[
\mathbb{E}[b^{\xi_k}\bar{\rho}_t^{\xi_k}-b\bar{\rho}_t]
=\E(b^{\xi_k}-b)(\bar{\rho}_t^{\xi_k}-\bar{\rho}_t)
\]
in \eqref{eq4_9}.

At last, we remark that though $\bar{\rho}_t$ is close to $\rho_t$, in practice we do not repeat the experiment for many times and take the average. Instead, we only generate a sequence of the batches $(\xi^0, \cdots, \xi^k, \cdots)$ and generate the samples. Even though we use a single $b^{\xi_k}$ each step, the empirical distribution by the generated samples converges weakly to the invariant measure with error $O(\eta)$ under Wasserstein distance to the interested distribution, due to the ergodicity of SGLD. 
Hence, one does not have to run SGLD for many experiments to approximate $\bar{\rho}_t$ and further to approximate the invariant distribution with error $O(\eta)$.

\item \textbf{Dependence on the dimensionality}

The linear scaling with respect to $d$ arising in \eqref{eq:klbatch} and \eqref{eqq319} is quite natural for the entropy and Fisher information. In fact, if one considers cases where the dynamics in difference dimensions are decoupled so that $\rho_t(x)=\prod_{i=1}^d\rho_t^{(i)}(x_i)$, the dependence on dimension in the entropy and Fisher information would be linear.  

As we have mentioned, the linear dependence in the our error bounds largely comes from with the factor $d^{-1/2}$ in the Lipschitz constant in Assumption \ref{ass:b}, and we have justified this below Assumption \ref{ass:b}.  A slight discrepancy of this intuition comes from $D_{KL}(\rho_0\|\pi)$, which by by condition (a) in Assumption \ref{ass:pi} satisfies
\[
D_{KL}(\rho_0\|\pi) \le \log \lambda.
\]
This independence of the dimension is a consequence of the strong assumption Assumption \ref{ass:pi}. However, if we use a weaker assumption like $\rho_0/\pi\le \lambda^d$, there would be dimension dependence in the constant for log-Sobolev inequality of the Hooley-Stroock perturbation lemma. This may indicate that the Hooley-Stroock perturbation lemma is not sharp regarding the scalability in $d$.

\end{itemize}

\section{Acknowledgement}
This work is financially supported by the National Key R\&D Program of China, Project Number 2021YFA1002800 and Project Number 2020YFA0712000. The work of L. Li was partially supported by NSFC 11901389 and 12031013,  Shanghai Municipal Science and Technology Major Project 2021SHZDZX0102, Shanghai Science and Technology Commission Grant No. 21JC1402900, and Shanghai Sailing Program 19YF1421300.

\begin{appendix}
\section{Omitted details in Section \ref{sec:mainresults} and \ref{sec:localpf}}\label{sgldpf}
In this section, we prove the details omitted in Section \ref{sec:mainresults} and \ref{sec:localpf}.  Lemma \ref{lmmrt}, \ref{integralbyparts} have been proved in \cite{mou2019improved}.

\subsection{Proof of Proposition \ref{thm:moment}}\label{omit_momentcontrol}
The method for bounding the $p$-th  moment of fixed-batch SGLD is based on direct It\^{o} calculation and basic inequalities. In the followings, we denote $\mathcal{F}_{\xi}$  the $\sigma$-algebra generated by $\xi_k$ for all $k \in \mathbb{N}$.
\begin{proof}[Proof of Proposition \ref{thm:moment}:]
By definition, for fixed batch sequence $\b{\xi}$, 
\begin{equation*}
    d\bar{X}_t = b^{\xi_k}(\bar{X}_{T_k})dt + \sqrt{2\beta^{-1}}dW, \quad t \in \left[T_k, T_{k+1} \right),
\end{equation*}
where we recall $T_k = \sum_{i=0}^{k-1} \eta_i$.
For $p \geq 2$, by It\^{o}'s formula, we have
\begin{multline*} 
    d|\bar{X}_t|^p  = p|\bar{X}_t|^{p-2} \bar{X}_t \cdot \left( b^{\xi}(\bar{X}_{T_k})dt+\sqrt{2\beta^{-1}}dW\right)\\
    +\beta^{-1}p|\bar{X}_t|^{p-2}\left(I_d + (p-2) \frac{\bar{X}_t\otimes\bar{X}_t}{|\bar{X}_t|^2} \right) : I_d dt.
\end{multline*}
So
\begin{multline}\label{eq33}
    \frac{d}{dt}\mathbb{E}\left[|\bar{X}_t|^p\Big|\mathcal{F}_{\xi}\right]\leq p\mathbb{E}\left[|\bar{X}_t|^{p-2} \bar{X}_t \cdot b^{\xi_k}(\bar{X}_t) \Big|\mathcal{F}_{\xi} \right]\\
    + p\mathbb{E}\left[|\bar{X}_t|^{p-2} \bar{X}_t \cdot \left( b^{\xi_k}(\bar{X}_{T_k}) - b^{\xi_k}(\bar{X}_t)\right) \Big|\mathcal{F}_{\xi} \right] + \beta^{-1}p(p-1)d\mathbb{E}\left[|\bar{X}_t|^{p-2}\Big|\mathcal{F}_{\xi}\right].
\end{multline}
Using the dissipation condition in Assumption \ref{ass:b}, we can control the first term above, namely,
\begin{equation*}
    p\mathbb{E}\left[|\bar{X}_t|^{p-2} \bar{X}_t \cdot b^{\xi_k}(\bar{X}_t) \Big|\mathcal{F}_{\xi}\right] \lesssim - p \mathbb{E}\left[|\bar{X}_t|^p\Big|\mathcal{F}_{\xi}\right] +  p \mathbb{E}\left[|\bar{X}_t|^{p-2}\Big|\mathcal{F}_{\xi}\right].
\end{equation*}
By Young's inequality, for any $\delta > 0$, 
\begin{equation*}
    \mathbb{E}\left[|\bar{X}_t|^{p-2}\Big|\mathcal{F}_{\xi}\right] \leq \delta d^{-1} \frac{p-2}{p} \mathbb{E}\left[|\bar{X}_t|^p\Big|\mathcal{F}_{\xi}\right] + \delta^{-\frac{p-2}{2}} d^{\frac{p-2}{2}}\frac{2}{p}
\end{equation*}
For the second term, direct estimate gives
\begin{equation*}
    \begin{aligned}
     & \quad \mathbb{E}\left[|\bar{X}_t|^{p-2} \bar{X}_t \cdot \left( b^{\xi_k}(\bar{X}_{T_k}) - b^{\xi_k}(\bar{X}_t)\right) \Big|\mathcal{F}_{\xi} \right] \leq c \mathbb{E} \left[|\bar{X}_t|^{p-1} |\bar{X}_t - \bar{X}_{T_k}| \Big|\mathcal{F}_{\xi} \right]\\
     & \lesssim (t-T_k)\mathbb{E}\left[|b^{\xi_k}(\bar{X}_{T_k})| |\bar{X}_t|^{p-1}  \Big|\mathcal{F}_{\xi}\right] + \sqrt{2\beta^{-1}}\mathbb{E}\left[ |\bar{X}_t|^{p-1}   |\int_{T_k}^t dW_s| \Big|\mathcal{F}_{\xi}\right]\\
     \end{aligned}
\end{equation*}
Note that $b(\cdot)$ grows linearly only. Then, by Young's inequality, together with the fact that $t-T_k \le \eta_k$, one has
\begin{multline*}
\mathbb{E}\left[|\bar{X}_t|^{p-2} \bar{X}_t \cdot \left( b^{\xi_k}(\bar{X}_{T_k}) - b^{\xi_k}(\bar{X}_t)\right) \Big|\mathcal{F}_{\xi}\right]\\
\lesssim (\eta_k^{p/(2(p-1))}+\eta_k)\mathbb{E}\left[|\bar{X}_t|^{p}\Big|\mathcal{F}_{\xi}\right]+\eta_k \mathbb{E} \left[|\bar{X}_{T_k}|^p\Big|\mathcal{F}_{\xi}\right]  + d^{\frac{p}{2}}\left(1 + \beta^{-\frac{p}{2}} \right).
\end{multline*}
Hence, for $t \in \left[T_k, T_{k+1} \right)$,
\begin{equation}
    \frac{d}{dt} \mathbb{E} \left[|\bar{X}_t|^p\Big|\mathcal{F}_{\xi}\right] \leq -c_1  \mathbb{E}\left[|\bar{X}_t|^p\Big|\mathcal{F}_{\xi}\right] + c_3  \eta_k \mathbb{E} \left[|\bar{X}_{T_k}|^p\Big|\mathcal{F}_{\xi}\right] + c_2d^{\frac{p}{2}}\left(1 + \beta^{-\frac{p}{2}} \right).
\end{equation}
Here, $c_1$, $c_2$, $c_3$ are positive constants independent of $t$, $d$ and $\b{\xi}$ but possibly dependent of $p$. 

We claim that since there exists $c_1'>0$ such that
\[
-c_1+c_3\eta_k \le -c_1',\quad  \forall k,
\]
the moment is uniformly bounded.

In fact, applying the Gr\"onwall inequality,
\[
\mathbb{E} \left[|\bar{X}_t|^p\Big|\mathcal{F}_{\xi}\right]
\le e^{-c_1(t-T_k)}\mathbb{E} \left[|\bar{X}_{T_k}|^p\Big|\mathcal{F}_{\xi}\right]
+\int_{T_k}^t e^{-c_1(t-s)}\left(c_3  \eta_k \mathbb{E} \left[|\bar{X}_{T_k}|^p\Big|\mathcal{F}_{\xi}\right] + c_2d^{\frac{p}{2}}\left(1 + \beta^{-\frac{p}{2}} \right)\right)ds.
\]
Defining
\[
u^{\b{\xi}}(t):=\sup_{0\le s\le t}\mathbb{E} \left[|\bar{X}_s|^p\Big|\mathcal{F}_{\xi}\right],
\]
one thus finds 
\begin{gather*}
u^{\b{\xi}}(t)\le  e^{-c_1(t-T_k)}u^{\b{\xi}}(T_k)
+\int_{T_k}^t e^{-c_1(t-s)}[c_3  \eta_k u^{\b{\xi}}(s)  + c_2d^{\frac{p}{2}}\left(1 + \beta^{-\frac{p}{2}} \right)]ds.
\end{gather*}
Then, it is not hard to find that $u^{\b{\xi}}(t)\le v^{\b{\xi}}(t)$
where
\begin{gather}\label{eq:compare}
\frac{d}{dt}v^{\b{\xi}}(t)= -c_1' v^{\b{\xi}}(t)+c_2d^{\frac{p}{2}}\left(1 + \beta^{-\frac{p}{2}} \right), \quad t\in[T_k, T_{k+1}), ~v^{\b{\xi}}(T_k)\ge u^{\b{\xi}}(T_k).
\end{gather}
\begin{remark}
One may use an intermediate function satisfying the following to justify the above comparison principle:
\[
\frac{d}{dt}\tilde{v}^{\b{\xi}}(t)=-c_1\tilde{v}^{\b{\xi}}(t)+c_3\eta_k \tilde{v}^{\b{\xi}}(t)+c_2d^{\frac{p}{2}}\left(1 + \beta^{-\frac{p}{2}} \right).
\]
\end{remark}
Since \eqref{eq:compare} holds for each time interval and the moment is continuous in time so one can concatenate them and
conclude by induction that $u^{\b{\xi}}(t)\le v^{\b{\xi}}(t)$ with
\[
\frac{d}{dt}v^{\b{\xi}}(t)= -c_1' v^{\b{\xi}}(t)+c_2d^{\frac{p}{2}}\left(1 + \beta^{-\frac{p}{2}} \right),\quad v(0)=u(0)=\E|\bar{X}_0|^p.
\]
Hence,
\begin{equation}
    \mathbb{E}\left[|\bar{X}_t|^p\Big|\mathcal{F}_{\xi}\right] \leq c_pd^{\frac{p}{2}}\left(1 + \beta^{-\frac{p}{2}} \right),\quad \forall t>0,
\end{equation}
where $c_p$ is a positive constant independent of $t$, $d$ and $\b{\xi}$ but possibly dependent of $p$.

For the exponential moment $\mathbb{E}\left[ e^{\alpha |X|^p}\right]$ with small $\alpha$, after similar It\^{o}'s calculation, we are able to obtain
\begin{equation}
    \sup_{t>0} \mathbb{E}\left[e^{\alpha|\bar{X}_t|^p}\Big|\mathcal{F}_{\xi}\right] < + \infty. 
\end{equation}
This then ends the proof.
\end{proof}

\subsection{Proof of Lemma \ref{Mfisher}}\label{omit_Mfisher}
\begin{proof}[Proof of Lemma \ref{Mfisher}:]

We first claim the followings:
\begin{itemize}
\item There exist positive constants $c_0$, $c_1 $ independent of the time $t$ such that
\begin{equation}\label{claimKLFisher}
    \frac{d}{dt} D_{KL}(\rho^{\b{\xi}}_t\|\pi) \leq -c_0 \beta^{-1} \mathcal{I}(\rho^{\b{\xi}}_t) + c_1d(\beta+1).
\end{equation}
\item There is a positive constant $c_2 $ independent of the time $t$ ($c_2$ is dependent on $D_{KL}(\rho_0\|\pi)$, which is a dimension-free positive constant due to Assumption \ref{ass:pi}) such that
\begin{equation}\label{claim2KLFisher}
    D_{KL}(\rho^{\b{\xi}}_t\|\pi) \leq c_2d(\beta+1)\beta^{-1}\kappa(\beta)^{-1}.
\end{equation}
\end{itemize}

Indeed, for the first claim \eqref{claimKLFisher}, using Fokker-Planck equation \eqref{FPhat} for $\rho^{\b{\xi}}_t$, we can directly calculate the following:
\begin{equation*}
    \begin{aligned}
    \frac{d}{dt}D_{KL}(\rho^{\b{\xi}}_t\|\pi) &= \int \left(1 + \log \frac{\rho^{\b{\xi}}_t}{\pi} \right)\partial_t \rho^{\b{\xi}}_t dx\\
    &= \int \left(\nabla \log \rho^{\b{\xi}}_t - \beta b\right)\cdot \left(\rho^{\b{\xi}}_t \hat{b}_t - \beta^{-1}\nabla \rho^{\b{\xi}}_t \right)dx,
    \end{aligned}
\end{equation*}
where we recall $b=-\nabla U=\beta^{-1}\nabla\log \pi$.
By Young's inequality, we have 
\begin{gather}\label{eq221}
    \begin{split}
    \frac{d}{dt}D_{KL}(\rho^{\b{\xi}}_t\|\pi) &\leq \left(\frac{\beta^{-1}}{4} \int |\nabla \log \rho^{\b{\xi}}_t|^2 \rho^{\b{\xi}}_t dx + \beta \int |\hat{b}_t|^2 \rho^{\b{\xi}}_t dx\right) -\beta^{-1}\int |\nabla \log \rho^{\b{\xi}}_t|^2 \rho^{\b{\xi}}_t dx\\
    & \quad + \frac{\beta}{2}\int (|b|^2+|\hat{b}_t|^2)\rho_t^{\b{\xi}}dx  + \left(\beta\int |b|^2\rho_t^{\b{\xi}}dx + \frac{\beta^{-1}}{4} \int |\nabla \log \rho^{\b{\xi}}_t|^2 \rho^{\b{\xi}}_t dx \right)\\
    &=-\frac{1}{2}\beta^{-1}\mathcal{I}(\rho^{\b{\xi}}_t) + \frac{3\beta}{2}\left(\mathbb{E}\left[|\hat{b}_t(\bar{X}_t)|^2\Big| \mathcal{F}_{\xi}\right] + \mathbb{E}\left[|b(\bar{X}_t)|^2\Big| \mathcal{F}_{\xi}\right]\right).
    \end{split}
\end{gather}
Using the Lipschitz assumption, the result for moment control Proposition \ref{thm:moment}, and Jensen's inequality, since $\eta_k\leq \Delta_0$,  we know that for $t\in [T_k,T_{k+1})$, 
\begin{equation}\label{eqa21}
    \begin{aligned}
    \mathbb{E}\left[|\bar{b}_t(\bar{X}_t)|^2\Big|\mathcal{F}_{\xi}\right] &= \int \rho^{\b{\xi}}_t(x) \left(\mathbb{E}\left[b(\bar{X}_{T_k})|\bar{X}_t=x,\mathcal{F}_{\xi}\right] \right)^2 dx\\
    &\leq \mathbb{E}\left[|b(\bar{X}_{T_k})|^2\Big|\mathcal{F}_{\xi}\right]\leq \mathbb{E}\left[\left(|b(0)| + |\bar{X}_{T_k}|\right)^2\Big|\mathcal{F}_{\xi}\right]\leq \tilde{c}_1(1 + \beta^{-1})d,\\
    \end{aligned}
\end{equation}
where $c_1$ is a time-independent positive constant according to Proposition \ref{thm:moment}. Hence by \eqref{eq221} and \eqref{eqa21},  the first claim \eqref{claimKLFisher} holds.

For the second claim \eqref{claim2KLFisher}, we first observe that, we can control the negative Fisher information $-\mathcal{I}(\bar{\rho}^{\b{\xi}}_t)$ by the negative relative information $-\mathcal{I}(\rho^{\b{\xi}}_t|\pi) := -\int |\nabla \log \frac{\rho^{\b{\xi}}_t}{\pi}|^2 \rho^{\b{\xi}}_t dx$ via Young's inequality, namely,
\begin{equation}
    \mathcal{I}(\rho^{\b{\xi}}_t|\pi) = \int |\nabla \log \frac{\rho^{\b{\xi}}_t}{\pi}|^2 \rho^{\b{\xi}}_t dx \leq 2\mathcal{I}(\rho^{\b{\xi}}_t) + 2\beta^2 \mathbb{E}\left[|b(\bar{X}_t)|^2\Big|\mathcal{F}_{\xi}\right].
\end{equation}
So we have
\begin{equation}\label{eq224}
    -\mathcal{I}(\rho^{\b{\xi}}_t) \leq -\frac{1}{2}\mathcal{I}(\rho^{\b{\xi}}_t|\pi) + \beta^2\mathbb{E}\left[|b(\bar{X}_t)|^2\Big|\mathcal{F}_{\xi}\right].
\end{equation}
Then, combining  \eqref{eq221}, \eqref{eqa21} and \eqref{eq224}, we have 
\begin{equation}
\begin{aligned}
    \frac{d}{dt}D_{KL}(\rho^{\b{\xi}}_t\|\pi) &\leq -\frac{1}{4}\beta^{-1}\mathcal{I}(\rho^{\b{\xi}}_t|\pi) + 2\beta\mathbb{E}\left[|b(\bar{X}_t)|^2\Big| \mathcal{F}_{\xi}\right] + \frac{3\beta}{2}\mathbb{E}\left[|\hat{b}_t(\bar{X}_t)|^2\Big|\mathcal{F}_{\xi}\right]\\
    &\leq -\frac{1}{4}\beta^{-1}\mathcal{I}(\rho^{\b{\xi}}_t|\pi) + c_1'd(\beta+1),
\end{aligned}
\end{equation}
where $c_1'$ is a positive constant. Since $\pi$ satisfies a Log-Sobolev inequality with constant $\kappa(\beta)$, we have
\begin{equation}
     \frac{d}{dt}D_{KL}(\rho^{\b{\xi}}_t\|\pi) \leq -\frac{1}{4\beta \kappa(\beta)} D_{KL}(\rho^{\b{\xi}}_t\|\pi) + c_1'd(\beta+1).
\end{equation}
By Gr\"onwall's inequality, we have
\begin{equation}
\begin{aligned}
    D_{KL}(\rho^{\b{\xi}}_t\|\pi) &\leq e^{-4\beta \kappa(\beta) t} D_{KL}(\rho_0\|\pi) + c_1' d(\beta+1)(4\beta \kappa(\beta))^{-1} (1 - e^{-4\beta \kappa(\beta) t})\\
    &\leq c_2d(\beta+1)\beta^{-1}\kappa(\beta)^{-1},
\end{aligned}
\end{equation}
Hence the second claim \eqref{claim2KLFisher} holds.

Now, equipped with the two claims \eqref{claimKLFisher} and \eqref{claim2KLFisher}, using integration by parts, we know that for all differential, nonnegative, non-increasing $f$, for any $T>0$,

\begin{equation*}
    \begin{aligned}
    &\quad \int_0^T e^{-A_0(T-s)} f(s) \mathcal{I}(\rho^{\b{\xi}}_s) ds\\ &\leq -\tilde{c}_0 \beta \int_0^T e^{-A_0(T-s)} f(s) \left(\frac{d}{ds}D_{KL}(\rho^{\b{\xi}}_s\|\pi) \right)ds + \tilde{c}_1 \beta(\beta + 1) d\int_0^T e^{-A_0(T-s)} f(s) ds\\
    &= \tilde{c}_1 \beta(\beta+1) d\int_0^T e^{-A_0(T-s)} f(s) ds -\tilde{c}_0 \beta f(T)D_{KL}(\rho^{\b{\xi}}_T\|\pi) + \tilde{c}_0 \beta e^{-A_0 T} f(0) D_{KL}(\rho_0\|\pi)\\
    &\quad + \tilde{c}_0 \beta \int_0^T e^{-A_0(T-s)}f'(s)D_{KL}(\rho^{\b{\xi}}_s\|\pi)ds + \tilde{c}_0A_0 \beta \int_0^T e^{-A_0(T-s)} f(s) D_{KL}(\rho^{\b{\xi}}_s\|\pi) ds\\
    &\leq \tilde{c}_1 \beta (\beta+1)d\int_0^T e^{-A_0(T-s)} f(s) ds + 0 + \tilde{c}_0\beta e^{-A_0T}f(0)D_{KL}(\rho_0\|\pi) + 0\\
    &\quad + \tilde{c}_0c_2 A_0 (\beta + 1)\beta^{-1}\kappa(\beta)^{-1}d \int_0^T e^{-A_0(T-s)} f(s) ds\\
    & :=  M_1 \beta D_{KL}(\rho_0\|\pi)f(0) e^{-A_0T}+ M_2d(\beta +1) (A_0\beta^{-1}\kappa(\beta)^{-1}+\beta)\int_0^T e^{-A_0(T-s)}f(s)ds.
    \end{aligned}
\end{equation*}
Here, $\tilde{c}_0$, $\tilde{c}_1$, $M_1$, $M_2$ are positive constants independent of $\xi$, $d$ and $\beta$. The last inequality is because the KL-divergence is non-negative, and $f$ is nonnegative and non-increasing.

Now, we have already obtained the desired estimation
\begin{multline}\label{eq229}
    \int_0^T e^{-A_0(T-s)}f(s)\mathcal{I}(\rho^{\b{\xi}}_s)ds
    \leq  M_1 \beta D_{KL}(\rho_0\|\pi)f(0) e^{-A_0T}\\
    + M_2d(\beta +1) (A_0\beta^{-1}\kappa(\beta)^{-1}+\beta)\int_0^T e^{-A_0(T-s)}f(s)ds
\end{multline}
for all differential, non-increasing, nonnegative function $f$. Since piecewise-constant functions can be approximated by differential functions, we know that \eqref{eq229} holds for all non-increasing, nonnegative, piecewise-constant function $f$. This then ends the proof.

\end{proof}

\subsection{Basics on path measure}\label{sec:pathmeasure}
In this section, we review some basics of path measure. Consider the following two SDEs in $\mathbb{R}^d$ with different drifts but the same diffusion $\sigma$:
\begin{equation}
\begin{split}
&  dX^{(1)} = b^{(1)}(X^{(1)}, x_0) dt + \sigma \cdot dW,\\
&  dX^{(2)} = b^{(2)}(X^{(2)}, x_0) dt + \sigma \cdot dW,
\end{split}
\quad X^{(1)}_0 =X^{(2)}_0 = x_0.
\end{equation}
Here, $W$ is a standard Brownian motion under the probability $\mathbb{P}$ (the same for the two SDEs), and $x_0\sim \mu_0$ is a common, but random, initial position. We allow the drifts to be possibly dependent on the initial position $x_0$ for our application.  For fixed time interval $[0,T]$, the two processes $X^{(1)}$ and $X^{(2)}$ naturally induce two probability measures in the path space $\mathcal{X} := C([0,T];\mathbb{R}^d)$, denoted by $P^{(1)}$ and $P^{(2)}$, respectively. To be more specific, for $[s,t] \subset [0,T]$, $A \in \mathcal{B}(\mathbb{R}^d)$, $P^{(i)}([s,t]\times A) = \mathbb{P}(X_{\tau}^{(i)} \in A, \tau \in [s,t])$, $i=1,2$. It is also obvious that the two probability measures $P^{(1)}$, $P^{(2)}$ are mutually absolutely continuous, so we are able to define the Radon-Nikodym derivative $\frac{dP^{(1)}}{dP^{(2)}} \in L^1(P^{(2)},\mathcal{X})$.

To obtain the formula for $\frac{dP^{(1)}}{dP^{(2)}} \in L^1(P^{(2)},\mathcal{X})$, we recall that the Girsanov transform \cite{girsanov1960transforming,durrett2018stochastic, dalalyan2012sparse,gao2023optimal} asserts that under the probability measure $\mathbb{Q}$ satisfying
\begin{multline*} 
    \frac{d\mathbb{Q}}{d\mathbb{P}}(X^{(2)}(\omega)) = \exp\Big(\int_0^T  (b^{(1)}  - b^{(2)})(X_s^{(2)}, x_0) \cdot G^{-1}\sigma \cdot  dW_s\\
    - \frac{1}{2} \int_0^T \left|  (b^{(1)} - b^{(2)}) (X_s^{(2)}, x_0) \cdot G^{-1}\sigma \right|^2 ds\Big),
\end{multline*}
with the matrix $G$ is defined by 
\begin{equation*}
    G := \left(\sigma \sigma^T\right)^{-1},
\end{equation*}
the law of $X^{(2)}$  is the same as the law of $X^{(1)}$ under $\mathbb{P}$. In other words, for any Borel measurable set $B\subset \mathcal{X}$,
\[
\E_{\mathbb{P}}[1_B(X^{(1)})]
=\E_{\mathbb{Q}}[1_B(X^{(2)})]
=\E_{\mathbb{P}}\left[1_B(X^{(2)})\frac{d\mathbb{Q}}{d\mathbb{P}}\right].
\]
Note that the expression of $\frac{d\mathbb{Q}}{d\mathbb{P}}(\omega)=g(X^{(2)}(\omega))$ where ($X_0:=X(0)$)
\begin{multline*}
    g(X) = \exp\Big(\int_0^T  \left(b^{(1)} - b^{(2)}) (X, X_0)\right) \cdot G^{-1} \cdot dX \\
    + \frac{1}{2} \int_0^T \left(b^{(2)} \cdot G^{-1} \cdot b^{(2)} - b^{(1)} \cdot G^{-1} \cdot b^{(1)} \right)(X, X_0) ds\Big),
\end{multline*}
Since $P^{(1)}
=(X^{(1)})_{\#}\mathbb{P}$ and $P^{(2)}
=(X^{(2)})_{\#}\mathbb{P}$ are the laws of $X^{(1)}$ and $X^{(2)}$ respectively, then one has
\[
P^{(1)}(B)=\mathbb{E}_{X\sim P^{(2)}}\left[\textbf{1}_B(X) g(X)\right]
\]
It follows that the Radon-Nikodym derivative is 
\[
\frac{dP^{(1)}}{dP^{(2)}}(X) = g(X).
\]
Hence, since $dX^{(2)} = b^{(2)}(X^{(2)},x_0)dt + \sigma \cdot dW$, we have derived that 
\begin{multline}\label{Girsanov528}
    \frac{dP^{(1)}}{dP^{(2)}}(X^{(2)}(\omega)) =\frac{d\mathbb{Q}}{d\mathbb{P}}(\omega)= \exp\Big(\int_0^T (b^{(1)} - b^{(2)})(X_s^{(2)}, x_0) \cdot G^{-1}\sigma \cdot  dW_s \\
    - \frac{1}{2} \int_0^T \left| (b^{(1)} - b^{(2)}) (X_s^{(2)}, x_0) \cdot G^{-1}\sigma \right|^2 ds\Big),
\end{multline}
which is a martingale under $\mathbb{P}$ and its natural filtration defined by $\mathcal{F}^{(2)}_t := \sigma(X_s^{(2)}, s\leq t)$, $t\in [0,T]$. Note that \eqref{Girsanov528} will be used in our proof.

Moreover, we can view the process $X^{(i)}$ as an identical mapping: $X^{(i)} = (X^{(i)}_t)_{0\leq t \leq T}$: $\Omega \rightarrow \mathcal{X}$. For fixed $t\in[0,T]$, $X_t^{(i)}$ can be viewed as a measurable mapping from $\Omega$ to $\mathbb{R}^d$ and one in fact has $X_t^{(i)}=\omega_t\circ X^{(i)}\in\mathbb{R}^d$, where $\omega_t: \mathcal{X}\to \mathbb{R}^d$ is the time marginal defined by $\omega_t(\gamma)=\gamma_t$.  Then the law of the solution $X_t^{(i)}$ at time $t$, namely, the time marginal of $P^{(i)}$, is the push forward measures $P_t^{(i)} := (X^{(i)}_t)_{\#}\mathbb{P}=(\omega_t)_{\#}P^{(i)}$, $\forall t \in [0,T]$. This means $P_t^{(i)}$ is a probability measure in $\mathbb{R}^d$, satisfying
\begin{equation}
    P_t^{(i)}(A) = \mathbb{P}(X_t^{(i)} \in A), \quad \forall A \in \mathcal{B}(\mathbb{R}^d), \quad i=1,2.
\end{equation}
Clearly, at any time $t$, $P_t^{(1)}$ and $P_t^{(2)}$ are mutually absolutely continuous, the Radon-Nikodym derivative $\frac{dP_t^{(1)}}{dP_t^{(2)}} \in L^1(P_t^{(2)},\mathbb{R}^d)$ is well defined. 

The following Lemma describes the relationship between the two Radon-Nikodym derivatives (of path measures and of push forward measures):

\begin{lemma}\label{lmm:generalpathmeasure}
Let $Q_1$, $Q_2$ be two probability distributions on $\mathcal{X}$, and $Q_2$ is absolutely continuous with respect to $Q_1$. Let $\phi:\mathcal{X}\rightarrow \mathbb{R}^d$ be a measurable mapping, and consider the push forward measure $\phi_{\#}Q_1$ and $\phi_{\#}Q_2$, denoted by $Q_{1,\phi}$ and $Q_{2,\phi}$, respectively. Then the Randon-Nikodym derivatives $\frac{dQ_{1,\phi}}{dQ_{2,\phi}}\in L^1(dQ_{2,\phi}, \mathbb{R}^d)$, $\frac{dQ_1}{dQ_2} \in L^1(dQ_2, \mathcal{X})$ are well-defined, and 
\begin{equation}
    \frac{dQ_{1,\phi}}{dQ_{2,\phi}}(x) = \mathbb{E}_{X\sim Q_2}\left[\frac{dQ_1}{dQ_2} | \phi(X)=x\right],\quad Q_2-a.e.
\end{equation}
\end{lemma}

\begin{proof}[Proof of Lemma \ref{lmm:generalpathmeasure}:] 
Using the definition of Radon-Nikodym derivative, it suffices to check that for any $A\in \mathcal{B}(\mathbb{R}^d)$,
\begin{equation}
    \mathbb{E}_{x\sim Q_{1,\phi}}\left[\textbf{1}_A(x) \right] = \mathbb{E}_{x\sim Q_{2,\phi}}\left[\textbf{1}_A(x)\mathbb{E}_{X\sim Q_2}\left[\frac{dQ_1}{dQ_2}(X) |\phi(X)=x \right]\right].
\end{equation}
Indeed, using the definition of push forward measure and Randon-Nikodym derivative, as well as the conditional expectation, we have
\begin{equation*}
    \begin{aligned}
    LHS &= \mathbb{E}_{X\sim Q_1}\left[\textbf{1}_A(\phi(X))\right]\\
    &= \mathbb{E}_{X\sim Q_2}\left[ \frac{dQ_1}{dQ_2}(X)\textbf{1}_A(\phi(X))\right]\\
    &= \mathbb{E}_{X\sim Q_2}\left[\textbf{1}_A(\phi(X))\mathbb{E}_{X\sim Q_2}\left[\frac{dQ_1}{dQ_2}(X) | \phi(X)\right] \right]\\
    &= RHS.
    \end{aligned}
\end{equation*}
This is what we want.

\end{proof}

Clearly, if we take $\phi=\omega_t$ for $t\in[0,T]$, the time projection mapping, then we conclude the following result for Radon-Nikodym of time marginals.
\begin{corollary}\label{coro:pathmeasure}
For $t\in[0,T]$, recall the definition of $P_t^{(i)}$ and $P^{(i)}$, $i=1,2$ above. Then the following holds: 
\begin{equation}\label{lmmpath}
    \frac{dP_t^{(1)}}{dP_t^{(2)}} (x) = \mathbb{E}_{X\sim P^{(2)}} \left[\frac{dP^{(1)}}{dP^{(2)}}(X) | X_t = x \right], \quad \mathbb{P}-a.e.
\end{equation}

\end{corollary}

Also note that Corollary \ref{coro:pathmeasure} here is very useful. For instance, if we combine the result \eqref{lmmpath} in Corollary \ref{coro:pathmeasure} with the Girsanov transform in \eqref{Girsanov528}, we can express the quotient of densities of two processes $\rho^{(1)}/\rho^{(2)}$ by one process $X^{(2)}$ alone. Meanwhile, the information coming from the process $X^{(1)}$ is stored in the drift term $b^{(1)}(\cdot)$, which is usually a deterministic function.  Then we only need to look into one of these two processes instead of both of them, and this can often simplify the analysis.

\subsection{Estimate of the remaining terms}
We estimate the remaining terms \[\mathbb{E}_{\xi_k}\mathbb{E}\left[|\bar{b}^{\xi_k}_t(\bar{X}_t) - b^{\xi_k}(\bar{X}_t)|^2\Big| \xi_k\right],\]
and 
\[\mathbb{E}_{\xi_k,\tilde{\xi}_k} \left[\int |b^{\xi_k}-b|^2 \frac{|\bar{\rho}_t^{\tilde{\xi}_k} - \bar{\rho}_t^{\xi_k}|^2}{\bar{\rho}_t^{\tilde{\xi}_k}} dx\right]\]
in Section \ref{sec:localpf}.  Note that in the followings we consider the behaviour of SGLD in a time subinterval $[T_k,T_{k+1})$, and recall that $\bar{\rho}^{\xi_k}$, $\bar{\rho}^{\tilde{\xi}_k}$ are defined in \eqref{eq:Lstar}.

\begin{lemma}\label{lmmrt}
Recall the definition of $\bar{r}_t(x)$ in \eqref{defrt}. Then under the setting of Proposition \ref{local_estimate}, there exists a positive constant $c$ independent of the choice of the batch $\xi_k$ such that for all $t \in \left[T_k, T_{k+1} \right)$, it holds that
\begin{equation}\label{rtbound}
    \mathbb{E} \left[|\bar{r}_t(\bar{X}_t)|^2 \Big| \xi_k\right] \leq cd \beta^{-2} (t - T_k)^2.
\end{equation}
\end{lemma}
\begin{proof}[Proof of Lemma \ref{lmmrt}:]
Since $\nabla b^{\xi}$ is $L_2$-Lipschitz, we know that
\begin{equation}
    |\bar{r}_t(x)| \leq L_2 \mathbb{E} \left[ |\bar{X}_{T_k} - \bar{X}_t|^2 | \bar{X}_t = x,  \xi_k \right].
\end{equation}
Hence by Jensen's inequality, and using the Lipschitz assumption for $b^{\xi_k}$ in Assumption \ref{ass:b}, we have
\begin{equation*}
    \begin{aligned}
    \mathbb{E} \left[|\bar{r}_t(\bar{X}_t)|^2 \Big| \xi_k\right] &\leq L_2^2 d^{-1} \mathbb{E}\left[\left|\mathbb{E} \left[ |\bar{X}_{T_k} - \bar{X}_t|^2 | \bar{X}_t = x \right]\right|^2\Big| \xi_k\right]\\
    &\leq L_2^2  d^{-1}\mathbb{E}\left[\mathbb{E} \left[ |\bar{X}_{T_k} - \bar{X}_t|^4 | \bar{X}_t = x \right] \Big| \xi_k\right]\\
    &= L_2^2  d^{-1} \mathbb{E}\left[|b^{\xi_k}(\bar{X}_{T_k})(t-T_k) + \sqrt{2\beta^{-1}}\int_{T_k}^t dW|^4 \Big| \xi_k\right]\\
    & \leq L_2^2 d^{-1} \left((t-T_k)^4\left(|b^{\xi_k}(0)| + L\mathbb{E}\left[|\bar{X}_{T_k}|\Big| \xi_k\right]\right)^4 + 96(t-T_k)^2d^2\beta^{-2}\right),
    \end{aligned}
\end{equation*}
where we use the inequality $(a+b)^p \leq 2^{p-1} (a^p + b^p)$ in the last inequality.

Finally, by Proposition \ref{thm:moment}, we have a uniform bound for the moment $\mathbb{E}\left[|\bar{X}_{T_k}|^4\Big| \mathcal{F}_{\xi}\right]$, which leads to the conclusion \eqref{rtbound}.
\end{proof}

\begin{lemma}\label{integralbyparts} Under the setting of Proposition \ref{local_estimate}, there exists a positive constant $c$ independent of $k$ and $\xi_k$ such that for $t\in [T_k,T_{k+1})$,  
\begin{equation*}
    \int \left|\mathbb{E} [\bar{X}_{T_k} - \bar{X}_t | \bar{X}_t  = x,\xi_k ]\right|^2 \bar{\rho}_t^{\xi_k}(x)dx \leq c\eta_k^2 \left(d(1 + \beta^{-1})+\beta^{-2}\mathcal{I}(\bar{\rho}_{T_k})\right),
\end{equation*}
where $\mathcal{I}(\rho) := \int \rho |\nabla \log \rho|^2 dx$ is the Fisher information.
\end{lemma}

\begin{proof}[Proof of Lemma \ref{integralbyparts}:]
By Bayes' law, we have
\begin{equation}
\begin{aligned}
     \mathbb{E} [\bar{X}_{T_k} - \bar{X}_t | \bar{X}_t  = x,\xi_k ] &= \int (y-x)P(\bar{X}_{T_k} = y | \bar{X}_t = x,\xi_k) dy\\
     & = \int (y-x) \frac{\bar{\rho}_{T_k}(y)P(\bar{X}_t = x | \bar{X}_{T_k} = y,\xi_k)}{\bar{\rho}_t^{\xi_k}(x)} dy.
\end{aligned}
\end{equation}
Since $P(\bar{X}_t = x | \bar{X}_{T_k} = y,\xi_k)$ is Gaussian, and since its derivative is also similar to the Gaussian form, we split the term $ \mathbb{E} [\bar{X}_{T_k} - \bar{X}_t | \bar{X}_t = x, \xi_k ]$ into three parts and use integration by parts to handle them.

Let
\begin{equation*}
    \begin{aligned}
    y - x &= \left(I_d + (t-T_k)  \nabla b^{\xi_k} (y) \right) \cdot \left(y - x + (t-T_k) b^{\xi_k}(y) \right)\\
    & \quad- (t-T_k)  \nabla b^{\xi_k}(y) \cdot \left(y - x + (t-T_k) b^{\xi_k}(y) \right)\\
    & \quad- (t-T_k) \cdot b^{\xi_k}(y)\\
    & := a_1(x,y) - a_2(x,y) - a_3(x,y),
    \end{aligned}
\end{equation*}
and define
\begin{equation}\label{Iidef}
    I_i(x) := \mathbb{E}\left[a_i( \bar{X}_t,\bar{X}_{T_k}) | \bar{X}_t = x \right], \quad i=1,2,3.
\end{equation}

Next, we will obtain an $\eta_k^2$ estimate for $\mathbb{E}\left[|I_i(\bar{X}_t)|^2\Big| \xi_k \right]$, $i = 1,2,3$. 
We first claim that, there exist positive constants $c$, $c'$, $c''$ independent of $k$ and $\xi$ such that for $t\in [T_k,T_{k+1})$, 
\begin{equation}\label{claimI1}
    \mathbb{E}\left[|I_1(\bar{X}_t)|^2\Big| \xi_k\right] \leq c\beta^{-2}(t-T_k)^2 \mathcal{I}(\bar{\rho}_{T_k}),
\end{equation}
\begin{equation}\label{claimI2}
    \mathbb{E}\left[|I_2(\bar{X}_t)|^2\Big| \xi_k\right] \leq c'd\beta^{-1}(t-T_k)^3,
\end{equation}
\begin{equation}\label{claimI3}
    \mathbb{E}\left[|I_3(\bar{X}_t)|^2\Big| \xi_k\right] \leq c''d(1 + \beta^{-1})(t-T_k)^2,
\end{equation}
where $\mathcal{I}(\rho) := \int \rho |\nabla \log \rho|^2 dx$ is the Fisher information.

By the claims \eqref{claimI1}, \eqref{claimI2}, \eqref{claimI3} , we can easily obtain the following $O\left(\eta_k^2 (d+\mathcal{I}(\bar{\rho}_{T_k})\right)$  bound:
\begin{equation*}
\quad\int |\mathbb{E} [\bar{X}_{T_k} - \bar{X}_t | \bar{X}_t  = x, \xi_k ]|^2 \bar{\rho}_t^{\xi_k}(x)dx \leq 3\sum_{i=1}^3 \mathbb{E} \left[|I_i(\bar{X}_t)|^2\Big| \xi_k \right] \leq \tilde{c} \eta_k^2  \left( d + \mathcal{I}(\bar{\rho}_{T_k})\right),
\end{equation*}
where the positive constant $\tilde{c}$ is independent of the batch $\xi_k$, and this is what we desire.

For the the claims \eqref{claimI1}, \eqref{claimI2}, \eqref{claimI3}, we prove them separately in the followings:

(a) For the term $I_1$, since the distribution $P(\bar{X}_t=x|\bar{X}_{T_k}=y, \xi_k)$ is Gaussian, namely,
\begin{equation}
    P(\bar{X}_t=x|\bar{X}_{T_k}=y, \xi_k) = \left(4\pi \beta^{-1} (t-T_k)\right)^{-\frac{d}{2}}\exp\left(-\frac{|x-y-b^{\xi_k}(y)(t-T_k)|^2}{4\beta^{-1}(t-T_k)}\right).
\end{equation}
Then, after integration by parts we obtain:
\begin{equation}
    I_1(x) = 2\beta^{-1}(t-T_k)\int  \frac{\nabla_y \bar{\rho}_{T_k}(y)}{\bar{\rho}_t^{\xi_k}(x)} P(\bar{X}_t = x | \bar{X}_{T_k} = y, \xi_k) dy.
\end{equation}
Using Bayes' law again, we have
\begin{equation}
    I_1(x) = 2\beta^{-1}(t-T_k) \int  \frac{\nabla_y \bar{\rho}_{T_k}(y)}{\bar{\rho}_{T_k}(y)} P(  \bar{X}_{T_k} = y|\bar{X}_t = x, \xi_k) dy.
\end{equation}
Hence, by Jensen's inequality,
\begin{equation}
\begin{aligned}
    \mathbb{E}\left[|I_1(\bar{X}_t)|^2\Big| \xi_k \right] &\leq 4\beta^{-2}(t-T_k)^2 \int \bar{\rho}_t^{\xi_k}(x) \int  \left|\frac{\nabla_y \bar{\rho}_{T_k} (y)}{\bar{\rho}_{T_k}(y)}\right|^2 P(  \bar{X}_{T_k} = y|\bar{X}_t = x ,\xi_k) dy dx\\
    & = 4\beta^{-2}(t-T_k)^2 \mathcal{I}(\bar{\rho}_{T_k}) := c\beta^{-2}(t-T_k)^2\mathcal{I}(\bar{\rho}_{T_k}).
\end{aligned}
\end{equation}

(b) For the term $I_2$, using the Lipschitz condition in Assumption \ref{ass:b} and Jensen's inequality, we have
\begin{equation}
\begin{aligned}
    \mathbb{E}\left[|I_2(\bar{X}_t)|^2\Big| \xi_k\right]   &\leq  \tilde{c}' (t - T_k)^2\mathbb{E}\left[\left|  \mathbb{E}\left[\bar{X}_t - \bar{X}_{T_k} - (t-T_k)b^{\xi_k}(\bar{X}_{t}) \Big| \bar{X}_{T_k} \right]\right|^2 \Big| \xi_k\right]   \\
    &\le \tilde{c}' (t - T_k)^2 \mathbb{E}\left[\left|  \bar{X}_t - \bar{X}_{T_k} - (t-T_k)b^{\xi_k}(\bar{X}_{t})\right|^2\Big| \xi_k\right]\\
    &\leq \tilde{c}' (t-T_k)^2\mathbb{E}\left[\left|\int_{T_k}^t \sqrt{2\beta^{-1}}dW_s\right|^2\Big| \xi_k\right] = c'd\beta^{-1}(t-T_k)^3 .
\end{aligned}
\end{equation}
The constant $c'$ here is independent of $\xi_k$ since the Lipschitz constant is uniform.

(c) For the term $I_3$, still using Jensen's inequality and Lipschitz assumption for $b^{\xi_k}$, it is clear that
\begin{equation}
\begin{aligned}
    \mathbb{E}\left[|I_3(\bar{X}_t)|^2\Big| \xi_k\right] &\leq (t-T_k)^2 \mathbb{E}\left[|b^{\xi_k}(\bar{X}_{T_k})|^2\Big| \xi_k\right]\\
    &\leq 2(t-T_k)^2 \left(|b^{\xi_k}(0)|^2 + L_2\mathbb{E}\left[|\bar{X}_{T_k}|^2\Big| \xi_k\right]\right).
\end{aligned}
\end{equation}
By Proposition \ref{thm:moment}, for any fixed batch $\xi_k$, the moment $\mathbb{E}\left[|\bar{X}_{T_k}|^2\Big| \xi_k\right]$ has a uniform-in-batch $O(1)$ upper bound.  Therefore, we are able to obtain the following batch-independent bound for the term $\bar{I}_3$:
\begin{equation}
    \mathbb{E}\left[|I_3(\bar{X}_t)|^2 \Big| \xi_k\right]\leq c''d(1 + \beta^{-1}) (t-T_k)^2.
\end{equation}

Concluding the results we obtained in parts (a), (b), (c) yields  the claims \eqref{claimI1}, \eqref{claimI2}, \eqref{claimI3}.

\end{proof}

In the following lemmas, we prove some details for estimate of $\mathbb{E}_{\xi_k,\tilde{\xi}_k} \left[\int |b^{\xi_k}-b|^2 \frac{|\bar{\rho}_t^{\tilde{\xi}_k} - \bar{\rho}_t^{\xi_k}|^2}{\bar{\rho}_t^{\tilde{\xi}_k}} dx\right]$.

\begin{lemma}\label{lmm:i}
Recall the notations 
\begin{equation}
    K_1 := \frac{\beta}{2}\,\tilde{b}(y)(x-y),\quad \tilde{b}(y) := (b^{\xi_k}-b^{\tilde{\xi}_k})(y).
\end{equation}
Then under the conditions of Proposition \ref{local_estimate}, there exists a positive constant $c$ independent of $k$ and $\xi_k$,$\tilde{\xi}_k$ such that for $t\in [T_k,T_{k+1})$, 
\begin{equation}
    \int \bar{\rho}_t^{\tilde{\xi}_k} (x) \left(\int K_1  P(\bar{X}_{T_k} = y|\bar{X}'_t = x,\xi_k,\tilde{\xi}_k) dy \right)^2dx \leq c\eta_k^2 \left(d(\beta^2+1)+\mathcal{I}(\bar{\rho}_{T_k})\right).
\end{equation}
\end{lemma}
\begin{proof}[Proof of Lemma \ref{lmm:i}:]
The technique here is similar to the one we use in the proof of Lemma \ref{integralbyparts}. To begin with, we split the term $K_1$ into the following three parts:
\begin{equation*}
    \begin{aligned}
    \tilde{b}(y) \cdot (y - x) &= \tilde{b}(y) \cdot \left(I_d + (t-T_k)  \nabla b^{\tilde{\xi}_k} (y) \right) \cdot \left(y - x + (t-T_k) b^{\tilde{\xi}_k}(y) \right)\\
    & \quad- (t-T_k) \tilde{b}(y) \cdot \nabla b^{\tilde{\xi}_k}(y) \cdot \left(y - x + (t-T_k) b^{\tilde{\xi}_k}(y) \right)\\
    & \quad- (t-T_k) \tilde{b}(y) \cdot b^{\tilde{\xi}_k}(y)\\
    & := \bar{a}_1(x,y) - \bar{a}_2(x,y) - \bar{a}_3(x,y),
    \end{aligned}
\end{equation*}
and define
\begin{equation*}
    \bar{I_i}(x) := \mathbb{E}\left[\bar{a}_i(\bar{X}_{T_k}, \bar{X}'_t) \Big| \bar{X}'_t = x, \xi_k,\tilde{\xi}_k  \right], \quad i=1,2,3.
\end{equation*}
Then,
\begin{equation}
   \int \bar{\rho}_t^{\tilde{\xi}_k} (x) \left(\int K_1  P(\bar{X}_{T_k} = y|\bar{X}'_t = x) dy \right)^2dx =\frac{\beta^2}{4}\mathbb{E} \left[|\bar{I_1}(\bar{X}'_t)- \bar{I_2}(\bar{X}'_t) -\bar{I_3}(\bar{X}'_t)|^2\Big| \xi_k,\tilde{\xi}_k \right]
\end{equation}

(a) For the first term $\bar{I_1}$, using Bayes' formula and integration by parts, since $P(\bar{X}'_t= x|\bar{X}_{T_k} = y ,\xi_k,\tilde{\xi}_k )$ is Gaussian, we have
\begin{equation*}
    \begin{aligned}
    \bar{I_1}(x) &= 2\beta^{-1}\int \tilde{b}(y) \cdot \left(I_d + (t-T_k)  \nabla b^{\tilde{\xi}_k} (y) \right) \cdot \frac{\beta}{2}\left(y - x + (t-T_k) b^{\tilde{\xi}_k}(y) \right) \frac{\bar{\rho}_{T_k}(y)}{\bar{\rho}_t^{\tilde{\xi}_k}(x)}  P(\bar{X}'_t = x|\bar{X}_{T_k} = y ) dy\\
    &=2\beta^{-1}(t-T_k) \int \frac{\nabla_y(\tilde{b}(y) \bar{\rho}_{T_k}(y))}{\bar{\rho}_t^{\tilde{\xi}_k}(x)}P(\bar{X}'_t= x|\bar{X}_{T_k} = y ,\xi_k,\tilde{\xi}_k ) dy\\
    &= 2\beta^{-1}(t-T_k)\int \left(\nabla \tilde{b}(y) + \tilde{b}(y) \frac{\nabla \bar{\rho}_{T_k}(y)}{\bar{\rho}_{T_k}(y)}\right) P(\bar{X}_{T_k} = y| \bar{X}'_t = x,\xi_k,\tilde{\xi}_k) dy. 
    \end{aligned}
\end{equation*}
Since $\nabla \tilde{b}$ and $\tilde{b}=[(b^{\xi_k} - b) - (b^{\tilde{\xi}_k} - b)]/\sqrt{2\beta^{-1}}$ are uniformly bounded by Assumption \ref{ass:b}, we obtain
\begin{equation}\label{433}
     \beta^2\mathbb{E} \left[| \bar{I_1}(\bar{X}'_t) |^2\Big| \xi_k,\tilde{\xi}_k\right] \leq c\eta_k^2\left(1+\mathcal{I}(\bar{\rho}_{T_k})\right),
\end{equation}
and the constant $c$ is independent of $\xi_k$ and $\tilde{\xi}_k$.

(b) For the second term $\bar{I}_2$, by Jensen's inequality, it holds that 
\begin{equation}
    \beta^2\mathbb{E}\left[|\bar{I}_2(\bar{X}'_t)|^2\Big| \xi_k,\tilde{\xi}_k\right] \leq \beta^2(t-T_k)^2\mathbb{E}\left[ \left|\tilde{b}(\bar{X}_{T_k}) \cdot \nabla b^{\tilde{\xi}_k}(\bar{X}_{T_k}) \cdot \int_{T_k}^t dW\right|^2\Big| \xi_k,\tilde{\xi}_k\right].
\end{equation}
By Assumption \ref{ass:b}, both $\tilde{b}=[(b^{\xi_k} - b) - (b^{\tilde{\xi}_k} - b)]$ and $\nabla b^{\tilde{\xi}_k}$ are uniformly bounded, so we can directly get
\begin{equation}
   \beta^2 \mathbb{E}\left[|\bar{I}_2(\bar{X}'_t)|^2\Big| \xi_k,\tilde{\xi}_k\right] \leq cd\beta^2(t-T_k)^3,
\end{equation}
and the constant $c$ is independent of $\xi_k$ and $\tilde{\xi}_k$.

(c) For the third term $\bar{I}_3$, by Jensen's inequality, 
\begin{equation}
   \beta^2 \mathbb{E}\left[\left|\bar{I}_3(\bar{X}'_t)\right|^2\Big| \xi_k,\tilde{\xi}_k\right]  \leq \beta^2(t-T_k)^2\mathbb{E}\left[\left|\tilde{b}(\bar{X}_{T_k}) \cdot b^{\tilde{\xi}_k}(\bar{X}_{T_k})\right|^2\Big| \xi_k,\tilde{\xi}_k\right].
\end{equation}
Since $\tilde{b}$ is bounded, and since $b^{\tilde{\xi}_k}$ is Lipschitz by Assumption \ref{ass:b}, we have
\begin{equation}
    \begin{aligned}
    \beta^2\mathbb{E}\left[|\bar{I}_3(\bar{X}'_t)|^2\Big| \xi_k,\tilde{\xi}_k \right] &\leq c\beta^2(t-T_k)^2\mathbb{E}\left[|b^{\tilde{\xi}_k}(\bar{X}_{T_k})|^2\Big| \xi_k,\tilde{\xi}_k\right]\\
    &\leq 2c\beta^2(t-T_k)^2 \left(|b^{\tilde{\xi}_k}(0)|^2 + L^2 \mathbb{E}\left[|\bar{X}_{T_k}|^2\Big| \xi_k,\tilde{\xi}_k\right] \right).
    \end{aligned}
\end{equation}
Due to \eqref{eq:bddiff}, $|b^{\tilde{\xi}_k}(0)|
\le |b(0)|+|b^{\tilde{\xi}_k}(0)-b(0)|$ is uniformly bounded almost surely. By Proposition \ref{thm:moment},  there is a uniform-in-batch $O(1)$ bound for the moment $\mathbb{E}|\bar{X}_{T_k}|^2$. Hence we obtain
\begin{equation}
    \beta\mathbb{E}\left[|\bar{I}_3(\bar{X}'_t)|^2\Big| \xi_k,\tilde{\xi}_k\right] \leq c d \beta^2(1 + \beta^{-1}) (t-T_k)^2,
\end{equation}
and the constant $c$ is independent of $\xi_k$ and $\tilde{\xi}_k$.

Finally, combining (a), (b)  and (c), we get the desired result.

\end{proof}

\begin{lemma}\label{lmm:ii}
Recall the notations 
\begin{equation}
    K_2 := \left(-\frac{\beta}{2}\,\tilde{b}(y) \cdot b^{\tilde{\xi}_k}(y)- \frac{\beta}{4} |\tilde{b}(y)|^2\right)(t-T_k),\quad \tilde{b}(y) := \left(b^{\tilde{\xi}_k} - b^{\xi_k}\right)(y).
\end{equation}
Then under the conditions of Proposition \ref{local_estimate}, there exists a positive constant $c$ independent of $k$ and $\xi_k$ such that for $t\in [T_k,T_{k+1})$, 
\begin{equation}
    \int \bar{\rho}_t^{\tilde{\xi}_k} (x) \left(\int K_2  P(\bar{X}_{T_k} = y|\bar{X}'_t = x,\xi_k,\tilde{\xi}_k) dy \right)^2dx \leq cd\beta(\beta+1)\eta_k^2.
\end{equation}
The bound is independent of the choice of the batches $\tilde{\xi}_k$ , $\xi_k$.
\end{lemma}
\begin{proof}[Proof of \ref{lmm:ii}:]
By Jensen's inequality, we have
\begin{equation}\label{eq450}
    \begin{aligned}
     &\quad\int \bar{\rho}_t^{\tilde{\xi}_k} (x) \left(\int K_2  P(\bar{X}_{T_k} = y|\bar{X}'_t = x) dy \right)^2dx\\
     &\leq  \int \bar{\rho}_t^{\tilde{\xi}_k} (x) \left(\int |K_2|^2  P(\bar{X}_{T_k} = y|\bar{X}'_t = x) dy \right)dx\\
     &=(t-T_k)^2\mathbb{E}\left[\left|\frac{\beta}{2}\,\tilde{b}(\bar{X}_{T_k})\cdot b^{\tilde{\xi}_k}(\bar{X}_{T_k}) + \frac{\beta}{4}|\tilde{b}(\bar{X}_{T_k})|^2\right|^2\Big| \xi_k,\tilde{\xi}_k\right]
    \end{aligned}
\end{equation}
Since $b^{\tilde{\xi}_k}$ is Lipschitz by Assumption \ref{ass:b}, we have
\begin{equation*}
\mathbb{E}\left[\left|\frac{\beta}{2}\,\tilde{b}(\bar{X}_{T_k})\cdot b^{\tilde{\xi}_k}(\bar{X}_{T_k}) + \frac{\beta}{4}|\tilde{b}(\bar{X}_{T_k})|^2\right|^2\Big| \xi_k,\tilde{\xi}_k\right]\leq \beta^2\left( \tilde{c}_1 + \tilde{c}_2 \mathbb{E}\left[|\bar{X}_{T_k}|^2\Big| \xi_k,\tilde{\xi}_k\right]\right).
\end{equation*}
Here, the positive constants $\tilde{c}_1$, $\tilde{c}_2$ are independent of the  batch since the coefficients in conditions (a), (d) of Assumption \ref{ass:b} are uniform-in-batch. By Proposition \ref{thm:moment}, $\mathbb{E}\left[|\bar{X}_{T_k}|^2\right]$ has a uniform-in-batch $O(1)$ upper bound. Combining this fact with \eqref{eq450}, one then has 
\begin{equation}
    \int \bar{\rho}_t^{\tilde{\xi}_k} (x) \left(\int K_2  P(\bar{X}_{T_k} = y|\bar{X}'_t = x,\xi_k,\tilde{\xi}_k) dy \right)^2dx \leq cd\beta^2(1 + \beta^{-1})\eta_k^2.
\end{equation}
Here, $c$ is a positive constant independent of $k$, $d$, $\xi_k$ and $\tilde{\xi}_k$.
\end{proof}

\begin{lemma}\label{lmm:iii}
Recall the notations
\begin{equation}
    K_3 := e^z - 1 - z,
\end{equation}
with 
\begin{equation}\label{def_z}
    z := \frac{\beta}{2}\,\tilde{b}(y)(x-y-(t-T_k)b^{\tilde{\xi}_k}(y)) - \frac{\beta}{4} |\tilde{b}(y)|^2(t-T_k),
\end{equation}
and
\begin{equation}
    \tilde{b}(y) := \left(b^{\tilde{\xi}_k} - b^{\xi_k}\right)(y).
\end{equation}
Then under the conditions of Proposition \ref{local_estimate}, there exists a positive constant $c$ independent of $k$ and $\xi_k$ such that for $t\in [T_k,T_{k+1})$, 
\begin{equation}
     \int \bar{\rho}_t^{\tilde{\xi}_k} (x) \left(\int K_3  P(\bar{X}_{T_k} = y|\bar{X}'_t = x,\xi_k,\tilde{\xi}_k) dy \right)^2dx \leq c d \beta^2 \eta_k^2.
\end{equation}
\end{lemma}

\begin{proof}[Proof of Lemma \ref{lmm:iii}:]
By Jensen's inequality, it holds that
\begin{equation}\label{331}
    \begin{aligned}
    &\quad\int \bar{\rho}_t^{\tilde{\xi}_k} (x) \left(\int K_3 P(\bar{X}_{T_k} = y|\bar{X}'_t = x,\xi_k,\tilde{\xi}_k) dy \right)^2dx\\
    & \leq \int \bar{\rho}_t^{\tilde{\xi}_k} (x) \int |K_3|^2  P(\bar{X}_{T_k} = y|\bar{X}'_t = x) dy dx
     = \mathbb{E}\left[\left|e^{\hat{Y}_t} - 1 - \hat{Y}_t \right|^2\Big|\xi_k,\tilde{\xi}_k\right],
    \end{aligned}
\end{equation}
where we denote the process 
\begin{equation}
\begin{split}
    \hat{Y}_t &:= \frac{\beta}{2}\tilde{b}(\bar{X}_{T_k}) \cdot \left( \bar{X}'_t -\bar{X}_{T_k}-(t-T_k)b^{\tilde{\xi}_k}(\bar{X}_{T_k}) \right) - \frac{\beta}{4} |\tilde{b}(\bar{X}_{T_k})|^2 \Delta t\\
    &= -\frac{\beta}{4}|\tilde{b}(\bar{X}_{T_k})|^2\Delta t + \sqrt{\frac{\beta}{2}}\tilde{b}(\bar{X}_{T_k})\cdot \int_{T_k}^t dW.
\end{split}
\end{equation}
with
\begin{equation}
    \Delta t := t - T_k, \quad \forall t \in [T_k,T_{k+1}).
\end{equation}

Denote $\bar{Z}_t := e^{\hat{Y}_t} - 1 - \hat{Y}_t$. In the following, we aim to obtain an $O(\eta_k^2)$ bound for the term $\mathbb{E}\left[|\bar{Z}_t|^2\Big| \xi_k,\tilde{\xi}_k \right]$ mainly via It\^{o}'s formula. In fact, for $t \in [T_k,T_{k+1})$
\begin{equation}
    \bar{Z}_t = \frac{\beta}{4}|\tilde{b}|^2 \Delta t + \sqrt{\frac{\beta}{2}}\tilde{b} \cdot \int_{T_k}^t \left(e^{\hat{Y}_s}-1\right)dW_s=\frac{\beta}{4}|\tilde{b}|^2 \Delta t + \sqrt{\frac{\beta}{2}}\tilde{b} \cdot \int_{T_k}^t \left(\bar{Z}_s+\hat{Y}_s\right)dW_s.
\end{equation}
So 
\begin{equation}
    \mathbb{E}\left[|\bar{Z}_t|^2 \Big| \xi_k,\tilde{\xi}_k \right] = \frac{\beta^2}{8}\mathbb{E}\left[ |\tilde{b}|^4\Big| \xi_k,\tilde{\xi}_k \right] (\Delta t)^2 + \beta \int_{T_k}^t\mathbb{E}\left[|\tilde{b}|^2\left(\bar{Z}_s+\hat{Y}_s\right)^2\Big| \xi_k,\tilde{\xi}_k\right]ds.
\end{equation}
By Assumption \ref{ass:b}, $\tilde{b}$ is uniformly bounded. Then, $\mathbb{E}[|\hat{Y}_t|^2 | \xi_k,\tilde{\xi}_k ]\le c d \beta \Delta t$. Then, one has
\begin{equation*}
\begin{aligned}
    \quad \mathbb{E}\left[|\bar{Z}_t|^2 \Big|\xi_k,\tilde{\xi}_k\right] &\leq  c\int_{T_k}^t \mathbb{E}\left[|\bar{Z}_s|^2 \Big| \xi_k,\tilde{\xi}_k\right]\,ds +cd\beta^{2}(\Delta t)^2.
\end{aligned}
\end{equation*}
By Gr\"onwall's inequality, when $\eta_k$ is small, one thus has
\begin{equation}
    \mathbb{E}\left[|\bar{Z}_t|^2 \Big| \xi_k,\tilde{\xi}_k \right] \leq cd(\Delta t)^2 \leq cd\beta^2\eta_k^2.
\end{equation}
Above, $c$ is a generative positive constant independent of $k$ and $\xi_k$, with the concrete meaning varying from line to line. This then ends the proof.


\end{proof}

\end{appendix}

\bibliographystyle{plain}
\bibliography{main}

\end{document}